\documentclass[a4paper,11pt]{amsart}
\usepackage{dsfont}
\usepackage{amsmath,amsthm,amssymb,tabu,stmaryrd}
\usepackage{mathtools}
\usepackage[all]{xy}
\usepackage[usenames, dvipsnames]{color}
\usepackage{listings} 
\usepackage{graphicx}
\usepackage[table]{xcolor}
\usepackage{hyperref}
\hypersetup{
	colorlinks=true,
	linkcolor=OliveGreen,
	filecolor=magenta,
	urlcolor=RoyalBlue,
	citecolor=OliveGreen
}
\usepackage[noadjust]{cite}
\usepackage{tikz,tikz-cd}
\usepackage{enumitem}
\usetikzlibrary{shapes.geometric,calc,arrows}
\usepackage{multirow}
\usepackage{skak}
\usepackage[normalem]{ulem}
\usepackage[makeroom]{cancel}
\usepackage{comment}

\usepackage[iso-8859-7]{inputenc}

\newtheorem{introductiontheorem}{Theorem}
\newtheorem{theorem}{Theorem}[section]
\newtheorem{corollary}[theorem]{Corollary}
\newtheorem{lemma}[theorem]{Lemma}
\newtheorem{proposition}[theorem]{Proposition}
\newtheorem{remark}[theorem]{Remark}
\newtheorem{definition}[theorem]{Definition}
\newtheorem{example}[theorem]{Example}

\newtheorem{notation}[theorem]{Notation}

\def\rank{\operatorname{rank}}

\def\NE{\operatorname{NE}}

\def\Exc{\operatorname{Exc}}

\def\Pic{\operatorname{Pic}}

\def\Proj{\operatorname{Proj}}

\def\Sing{\operatorname{Sing}}

\def\Aut{\operatorname{Aut}}
\def\Autzero{\mathrm{Aut}^\circ}

\def\PGL{\operatorname{PGL}}
\def\GL{\operatorname{GL}}
\def\PSO{\operatorname{PSO}}
\def\SO{\operatorname{SO}}
\def\Orth{\operatorname{O}}

\def\Cr{\operatorname{Cr}}
\def\cha{\operatorname{char}}

\newcommand{\OO}{\mathcal{O}}
\newcommand{\cQ}{\mathcal{Q}}

\renewcommand{\AA}{\mathbb{A}}
\newcommand{\NN}{\mathbb{N}}
\newcommand{\ZZ}{\mathbb{Z}}
\newcommand{\QQ}{\mathbb{Q}}
\newcommand{\RR}{\mathbb{R}}
\newcommand{\CC}{\mathbb{C}}
\newcommand{\PP}{\mathbb{P}}
\renewcommand{\FF}{\mathbb{F}}
\newcommand{\GG}{\mathbb{G}}
\newcommand{\kk}{\textbf{k}}
\newcommand{\KK}{\mathbb{K}}

\newcommand{\isom}{\cong}
\newcommand{\defeq}{\vcentcolon=}
\newcommand{\map}{\smash{\xymatrix@C=0.5cm@M=1.5pt{ \ar[r]& }}}
\newcommand{\rmap}{\smash{\xymatrix@C=0.5cm@M=1.5pt{ \ar@{-->}[r]& }}}
\newcommand{\psmap}{\smash{\xymatrix@C=0.5cm@M=1.5pt{ \ar@{..>}[r]& }}}

\definecolor{myorange}{HTML}{eb811a}

\title{Umemura quadric fibrations and maximal subgroups of $\boldsymbol{\Cr_n(\CC)}$}
\author{Enrica Floris, Sokratis Zikas}
\date{\today}
\address{}
\email{}

\begin{document}

\maketitle

\begin{abstract}
	We study the equivariant geometry of special quadric fibrations, called Umemura quadric fibrations, as well as the maximality of their automorphism groups inside $\Cr_n(\CC)$.
	We produce infinite families of pairwise non-conjugate maximal connected
	algebraic subgroups of $\Cr_n(\CC)$.
\end{abstract}

\section{Introduction}

We work over an algebraically closed field $\kk$ of characteristic $0$.

The \emph{Cremona group} $\Cr_n(\kk)$ is the group of birational transformations of the projective space $\PP^n$ over a field $\kk$.
A classical problem in the theory of Cremona groups is the classification of their \emph{maximal connected algebraic subgroups}, that is, subgroups acting rationally on $\PP^n$ that are maximal with this property (see \cite[Definition 1.1]{Dem} for the precise definition of a rational action).
While maximal connected algebraic subgroups and their classification are objects of interest by themselves, they also admit a nice geometric description:
they correspond to automorphism groups of ``highly symmetric'' rational varieties.

Enriques \cite{Enriques} classified maximal connected algebraic subgroups of $\Cr_2(\kk)$ and showed that they are conjugate to $\Autzero(S)$, where $S = \PP^2$ or $\FF_n$ with $n\neq 1$.
In dimension $3$ a similar classification was obtained by Umemura in a series of four papers \cite{Umemura80,Umemura82a,Umemura82b,Umemura85}.

Blanc, Fanelli and Terpereau  in \cite{BFT1, BFT2} used an approach based on the Minimal Model Program MMP to recover most of the classification of Umemura. 
More precisely, if $G$ is a connected subgroup of $\Cr_n(\kk)$ acting rationally on $\PP^n$, then by Weil's regularisation theorem there is a rational variety $Z$ such that $G$ acts regularly on $Z$.
After running an MMP on $Z$, we obtain a Mori fibre space $X\to S$ with $X$ rational, together with a regular action of $G$ on $X$.

The group $G$ is maximal if for every $G$-equivariant Sarkisov program from $X/S$ to another Mori fibre space $Y/T$, the group $\Autzero(Y)$ coincides with $G$ (see \cite{Flo20}).
The groups appearing in the classification of maximal subgroups of $\Cr_n(\kk)$ are thus automorphism groups of Mori fibre spaces.

In total contrast to the case of dimension $2$ and $3$, where we have a full classification, and apart from some reduction results in dimension $4$ (see \cite{BF}) there is no classification theory of maximal connected algebraic subgroups of $\Cr_n(\kk)$ for $n\geq 4$.
In this paper we aim at initiating a study of maximal subgroups in higher dimensions.

It is worth noting that, while in dimension  $2$ and $3$ as a byproduct of the classification every subgroup is contained in a maximal one, this is not true in dimension $n\geq 4$ by \cite{FFZ23, Koll24}.

In this note we produce many examples of pairwise non-conjugated maximal subgroups of $\Cr_n(\kk)$ for $n\geq 3$.
More specifically, we study a certain class of $n$-dimensional quadric fibrations, called \emph{Umemura quadric fibrations}, and give necessary and sufficient criteria for when their automorphism groups are maximal in $\Cr_n(\kk)$.

Denote by $\mathcal{E}_a$ the vector bundle $\OO_{\PP^1}^{\oplus n}\oplus \OO_{\PP^1}(-a)$ over $\PP^1$, let $g \in \mathbb \kk[t_0, t_1]$ be a homogeneous polynomial of degree $2a$.
Then the Umemura quadric fibration associated to $g$ is the following divisor inside $\PP(\mathcal{E}_a)$:
\[
	\cQ_g \defeq \left\{x_1^2 - x_0x_2 + x_3^2 + \dots + x_{n-1}^2 + g(t_0,t_1)x_n^2 = 0
	\right\}
	\subset \PP(\mathcal{E}_a).
\]

See Definition \ref{def:umemuraQuadric} for more details.
The restriction $\pi\colon \cQ_g\to \PP^1$ of the projective bundle structure is a Mori fibre space.
If $g$ has more than two roots, then $\Autzero(\mathcal Q_g)\cong \SO_n(\kk)$ (see Proposition \ref{prop:automorphismGroups} for a complete discussion of the automorphism group).

Automorphism groups of Umemura quadric fibrations occupy a remarkable place in the theory of maximal subgroups
in dimension $3$.
They appear in the classification of Umemura and Blanc-Fanelli-Terpereau and are the only ones lying in continuous families.
In dimension $4$, they serve as natural candidates for producing maximal subgroups from the point of view of \cite{BF}.
Moreover, finite subgroups of them have been utilized by Krylov to produce surprising results in the theory of finite subgroups of the Cremona group $\Cr_3(\kk)$ \cite{Krylov}.

Our main result is the following:
\begin{introductiontheorem}[= Theorem \ref{thm:mainTheorem}]
	Let $g \in \mathbb \kk[t_0, t_1]$ be a homogeneous polynomial of degree $2a$ and $\pi\colon \cQ_g \to\PP^1$ the associated Umemura quadric fibration.
	Write $g = f^2h$, where $f, h \in \kk[t_0,t_1]$ are homogeneous polynomials with $h$ being square-free. 
	Then $\Autzero(\cQ_g)$ is conjugated to a subgroup of $\Autzero(\cQ_h)$.
	
	Moreover 
	if $h$ and $h'$ are two square-free polynomials, we have:
	\begin{enumerate}
	\item $\Autzero(\cQ_h)$ is a maximal connected algebraic subgroup of $\Cr_n(\kk)$
	if and only if $h$ is constant or has at least 4 roots;
	\item $\Autzero(\cQ_h)$ and $\Autzero(\cQ_{h'})$ are conjugate if and only if $ h(t_0, t_1) = h'(\alpha(t_0, t_1))$, with $\alpha\in \PGL_2(\kk)$.
	\end{enumerate}
\end{introductiontheorem}

The outline of the paper is as follows:
in section \ref{sec:preliminaries} we collect some preliminary results that will be used thoughout the paper;
in section \ref{sec:umemuraQuadrics} we introduce Umemura quadric fibrations, compute their automorphism groups and analyze their equivariant geometry;
finally, in section \ref{sec:maximality}, we study birational relations among Umemura quadric fibrations and determine maximality of their automorphism groups in $\Cr_n(\kk)$.

\medskip

\noindent{\bf Acknowledgements.} We are thankful to Fabio Bernasconi, J\'er\'emy Blanc, Andrea Fanelli, Tiago Duarte Guerreiro, Erik Paemurru and Susanna Zimmermann for the useful discussions and suggestions.
We would like to thank the anonymous referee for their suggestions which greatly improved the quality of the exposition.
The second author acknowledges the support of the Swiss National Science Foundation Grant \textit{``Cremona groups of higher rank via the Sarkisov program"} P500PT\_211018 and would like to thank the Laboratoire de Math\'ematiques et Applications de l'Universit\'e de Poitiers for their hospitality.

\section{Preliminary results}\label{sec:preliminaries}

This section contains some preliminary definitions and results on the geometry of birational maps of varieties with terminal singularities. 
We refer to \cite{KM98} for the basic notions of the MMP.

\subsection{Extremal Divisorial Contractions}

	\begin{definition}\label{def:extremalDivisorialContraction}
		Let $Y$ be a variety with terminal singularities and $\Gamma \subset Y$ an irreducible subvariety of codimension at least $2$.
		An \emph{extremal divisorial contraction} is a birational morphism
		\[
		f\colon E \subset X \to \Gamma \subset Y
		\]
		such that:
		\begin{itemize}
			\item $X$ is $\QQ$-factorial and has terminal singularities;
			\item $f|_{X \setminus E}$ an isomorphism and $E$ a prime divisor;
			\item $-K_X$ is $f$-ample; 
			\item $\rho(X/Y) = 1$.
		\end{itemize}
	\end{definition}

Typical examples of extremal divisorial contractions are blowups as well as an infinite family of weighted blowups.
However Remark \ref{rem:termExtrFromSmoothPt} shows these are not the only examples, even when the center is a smooth point.
The situation is different if the centre is an orbit of codimension $2$ and $f$ is equivariant with respect to some group $G$, as in this case by \cite[Proposition 2.4]{BF21} we have only blow-ups.

\begin{example}\label{ex:stdWBlowup}
	Consider the standard $(1,\dots,1,b)$-weighted blowup of $0 \in \AA^{n+1}_{x_i,t}$
	\[
	\Proj\left(\oplus_{k\geq 0} \mathcal{I}_k\right) \to \AA^{n+1},
	\]
	where $\mathcal{I}_k = 
	\big(
	\left\{
	x_0^{m_0}\cdot\ldots \cdot x_{n-1}^{m_{n-1}}\cdot t^{m_n} \, | \, m_0 + \dots + m_{n-1} + bm_n \geq k 
	\right\}
	\big)$.
	This can also be described as 
	\[
	\begin{array}{ccc}
		X \defeq \AA^{n+1}/\GG_m & \to & \AA^n\\
		(u:x_0:\ldots:x_{n-1}:t) & \mapsto & (ux_0,\dots,ux_{n-1},u^bt),
	\end{array}
	\]
	where $\GG_m$ acts linearly with weights $(-1,1,\dots,1,b)$.
	We demonstrate how to extract the valuation of $E = \{u=0\}$ using the \emph{tower construction} \cite[Construction 3.1]{Kawakita}.
	
	Denote by $U$ the open subset $\{x_0 = 1\} \subset X$.
	Since $E \mapsto 0 \in \AA^n$, the first step of the tower construction is the blowup of $0 \in \AA^n$, locally described by
	\[
	\begin{array}{ccc}
		V_1 = \AA^{n+1}_{v_1,x_i,t} & \to &  \AA^{n+1}\\
		(v_1,x_1,\dots,x_{n-1},t) & \mapsto & (v_1,v_1x_1,\dots,v_1x_{n-1},v_1t)\\
		E_1 \defeq \{v_1 = 0\} & \mapsto & 0.
	\end{array}
	\]
	The induced birational map between $V_1$ and $U$ is given by
	\[
	(u,x_1,\dots,x_{n-1},t) \mapsto (u,x_1,\dots,x_{n-1},ut)
	\]
	and $E \mapsto \Gamma_1 \defeq \{v_1 = t =0\}$.
	The second step is the blowup of $V_1$ along $\Gamma_1$, locally described by
	\[
	\begin{array}{ccc}
		V_2 = \AA^{n+1}_{v_2,x_i,t} & \to & V_1\\
		(v_2,x_1,\dots,x_{n-1},t) & \mapsto & (v_2,x_1,\dots,x_{n-1},v_2t)\\
		E_2 \defeq \{v_2\} & \mapsto & \Gamma_1.
	\end{array}
	\]
	Again, the induced birational map between $U$ and $V_2$ is given by
	\[
	(u,x_1,\dots,x_{n-1},t) \mapsto (u,x_1,\dots,x_{n-1},ut).
	\]
	Continuing like that, we get the diagram
	\begin{equation*}
		\begin{aligned}
			\xymatrix@R=.4cm@C=1.8cm{
				&V_a \ar[d] \ar@/_1pc/@{-->}[lddd]_{\phi} \\
				&\vdots \ar[d]\\
				&V_1 \ar[d]\\
				U \ar[r]^f			&\AA^{n+1},
			}
		\end{aligned}
	\end{equation*}
	where, for $1\leq i < b$, $E_{i+1} \subset V_{i+1} \to \Gamma_i \subset V_i$ is the blowup of $V_i$ along $\Gamma_i$, with $\Gamma_i$ being the intersection of $E_i$ with the strict transform of $\{t=0\} \subset \AA^{n+1}$.
	
	We will return to this example again in Proposition \ref{prop:termExtractionFromP}.
\end{example}

\subsection{Action on a product of quadrics}

Let $Q_n\subseteq \PP^{n+1}$ be a smooth hypersurface of degree 2.  
In this subsection we recall some basic facts on the action of $\Aut^{\circ}(Q_n)$ on $Q_n\times Q_n$.

\begin{lemma}\label{lem:stabQ}
	Let $n\geq 3$ be an integer.
	Let $Q_n\subseteq \PP^{n+1}$ be a smooth hypersurface of degree 2.  
	Let $x\in Q_n$ be a point and $G_x$ the stabiliser of $x$ in $\Aut^{\circ}(Q_n)$.
	Let $H$ be the cone over a quadric of dimension $n-2$ obtained as intersection of $Q_n$ with the projective tangent space in $x$.
	The orbits of $G_x$ on  $Q_n$ are
	\begin{itemize}
		\item the point $\{x\}$;
		\item the $\AA^1$-bundle $H\setminus\{x\}$;
		\item the open set $Q_n\setminus H$.
	\end{itemize}
\end{lemma}

\begin{proof}
	Without loss of generality we may assume that $Q_n = \{x_0x_1 + x_2^2 + \dots + x_{n+1}^2 = 0\}$ and $x = (1:0:\ldots:0)$.
	Consider the birational map 
	\[
	\begin{array}{ccc}
		Q_n & \dasharrow & \PP^n\\
		(x_0:x_1:\ldots:x_{n+1}) & \mapsto & (x_1:\ldots:x_{n+1})\\
		\left(\frac{x_2^2 + \ldots + x_{n+1}^2}{x_1}:x_1:\ldots:x_{n+1}\right) & \mapsfrom &  (x_1:\ldots:x_{n+1}),
	\end{array}
	\]
	with its resolution $(p,q)\colon W \to Q_n \times \PP^n$, where:
	\begin{enumerate}
		\item\label{item 1} $p\colon W \to Q_n$ is the blowup along $x$;
		\item\label{item 2} $q\colon W \to \PP^n$ is the contraction of the strict transform of $H$. 
	\end{enumerate}
	By (\ref{item 1}) $\Autzero(W) \cong G_x$.
	Moreover $q$ coincides with the blowup of $\PP^n$ along $Q_{n-2} \defeq \{x_1 = x_2^2 + \dots + x_{n+1}^2 = 0\}$, and thus 
	\[
	pG_xp^{-1} = \Autzero(W) = q\Autzero(\PP^n;Q_{n-2})q^{-1}.
	\]
	The group $\Autzero(\PP^n;Q_{n-2})$ consists of matrices of the form 
	\[
	\left(
	\begin{array}{c|c}
		\alpha& 0\ldots 0\\
		\hline
		&\\[-8pt]
		b&B\\[-8pt]
		&
	\end{array}
	\right)
	\]
	where $\alpha\neq 0$ and $B^\tau B=\lambda I_{n-2}$.
	The orbits of  $\Autzero(\PP^n;Q_{n-2})$ are $Q_{n-2}$, $\{x_1 = 0\}\setminus Q_{n-2}$ and $\PP^n \setminus\{x_1 = 0\} $.
	The orbits of $\Autzero(W)$ are the intersection $Z$ of the exceptional divisor  $E$ of $q$ and the strict transform $P$ of $\{x_1 = 0\}$, the complements $P\setminus Z$ and $E\setminus Z$ and $W\setminus (E\cup P)$. Finally, the orbits of $G_x$ in $Q_x$ are the images of the orbits of $\Autzero(W)$ on $W$. Therefore the claim follows.
\end{proof}

The case $n=2$ is similar, yet slightly different.

\begin{lemma}\label{lem:stabQdim2}
	Let $Q_2\subseteq \PP^3$ be a smooth hypersurface of degree 2.  
	Let $x\in Q_2$ be a point and $G_x$ the stabiliser of $x$ in $\Aut^{\circ}(Q_2)$.
	The intersection of $Q_2$ with the projective tangent space at $x$ is the union of two lines $l_1, l_2$.
	The orbits of $G_x$ on  $Q_n$ are
	\begin{itemize}
		\item the point $\{x\}$;
		\item the lines $l_1\setminus\{x\}$ and $l_2\setminus\{x\}$;
		\item the open set $Q_2\setminus (l_1\cup l_2)$.
	\end{itemize}
\end{lemma}

\begin{proof}
	The proof is \textit{mutatis mutandis} the one of Lemma \ref{lem:stabQ}, the \textit{caveat} being that in this case $Q_{n-2} = Q_0$ is the union of $2$ distinct points $P_1,P_2$.
	Consequently, the group $\Aut(\PP^2;Q_0)$ is not connected.
	Restricting to the connected component containing the identity we get the claimed orbits. 
\end{proof}

\begin{lemma}\label{lem:actQxQ}
	Let $n\geq 2$ be an integer.
	Let $Q_n\subseteq \PP^{n+1}$ be a smooth hypersurface of degree 2. 
	Let $G=\Autzero(Q_n)$ and consider the diagonal action of $G$ on $Q_n\times Q_n$.
	Let $\Delta\subseteq Q_n\times Q_n$ be the diagonal and
	\[
	T \defeq \left(Q_n\times Q_n\right) \cap TQ_n \subset \PP^{n+1} \times \PP^{n+1},
	\]
	where $TQ_n$ denotes the projectivized tangent bundle of $Q_n$.
	\begin{itemize}
		\item If $n\geq 3$, the orbits of the action of $G$ on $Q_n\times Q_n$ are: $\Delta, T\setminus \Delta, Q_n\times Q_n\setminus T $;
		\item If $n=2$, $T=T_1\cup T_2$ and the orbits of the action of $G$ on $Q_n\times Q_n$ are: $\Delta, T_1\setminus \Delta, T_2\setminus \Delta, Q_n\times Q_n\setminus T $.
	\end{itemize}
\end{lemma}

\begin{proof}
	Since the action of $G$ on $Q_n$ is transitive, $\Delta$ is an orbit.
	
	We show the statement for  $n\geq 3$.
	Let $(x,y), (x',y')\in T$.
	Let $g\in G$ be such that $g(x')=x$. Then $g(y')\in T_{gx'} Q_n$.
	By Lemma \ref{lem:stabQ} there is $h\in G_{g(x)}$ such that $h(g(y'))=y$, and so $(hg)\cdot(x,y) = (x',y')$. 
	The proof that $Q_n\times Q_n\setminus T $ is an orbit is similar.
	
	We now treat the case $n=2$.
	By Lemma \ref{lem:stabQdim2}, for any $x \in Q_n$, the fiber over $x$ of the projection $T \to Q_n$ to the first factor has two irreducible components, namely the two $1$-dimensional orbits of $G_x$
	Let $(x,y) \in T$, $x' \in Q_2$ and $h \in G$ such that $h(x) = x'$.
	Then for every $g\in G$ with $g(x) = x'$, $h(y)$ and $g(y)$ lie in the same irreducible component: 
	indeed $hg^{-1} \in G_{x'}$ and $hg^{-1}(g(y)) = h(y)$.
	Thus the orbit of $(x,y)$ is a proper subset of $T$, that has the same dimension as its complement.
	If we denote by $T_1$ the closure of $G\cdot(x,y)$ and $T_2$ the closure of its complement we have $T = T_1 \cup T_2$.
	The rest of the proof is verbatim the proof of the higher dimensional case of the previous step.
\end{proof}

\subsection{A useful lemma} We end the section with the following lemma, which is well-known to experts. We give the proof for the readers convenience.

\begin{lemma}\label{lem:proj not product}
Let $Z$ be a smooth projective variety and $\mathcal L_i$ for $i=1,2$ line bundles on $Z$ such that $\mathcal L_1\not\equiv\mathcal L_2$. 
Denote by $P$ the projectivization $\PP_Z(\mathcal L_1 \oplus \mathcal L_2)$ together with the induced morphism $p\colon P \to Z$ and by $Z_i$ the sections of $p$ induced by $\mathcal L_i\to \mathcal L_1\oplus\mathcal L_2$.
Then every section of $p$ meets either $Z_1$ or $Z_2$.

If moreover $Z$ has Picard rank one and $\mathcal{L}_1^{\vee}\otimes \mathcal{L}_2$ is anti-ample, then $\NE(P)=\RR_+[f]+i_*(\NE(Z_2))$ where $f$ is a fibre of $p$ and $i\colon Z_2\to P$ is the immersion.
\end{lemma}

\begin{proof}
Let $\widetilde Z$ be a section of $p$.
There are divisors $D_1$, $D_2$ on $Z$ such that
\[
\widetilde Z\sim Z_1+p^*D_1 \;\;\;\;\;\;\;\;\; Z_2\sim Z_1+p^*D_2.
\]
Assume that $\widetilde Z$ is disjoint from $Z_1\cup Z_2$.
Then $$0=Z_2\vert_{\widetilde Z}=Z_1\vert_{\widetilde Z}+p^*D_2\vert_{\widetilde Z}=p^*D_2\vert_{\widetilde Z}.$$
Thus $D_2\sim 0$, implying that $Z_1$ and $Z_2$ are linearly equivalent.
But this would imply  that the linear equivalence classes of $\mathcal O_P(1)\vert_{Z_1}$ and $\mathcal O_P(1)\vert_{Z_2}$
are the same. 
Indeed, let $D$ be a divisor such that $\mathcal O_P(1)\sim \mathcal O_P(Z_1+p^*D)$. 
Thus $\mathcal O_P(1)\sim \mathcal O_P(Z_1+p^*D)\sim \mathcal O_P(Z_2+p^*D)$. 
Taking restrictions to $Z_1$ and to $Z_2$, we get $\mathcal O_P(1)\vert_{Z_1}\sim \mathcal O_P(p^*D)\vert_{Z_1}$ and $\mathcal O_P(1)\vert_{Z_2}\sim \mathcal O_P(p^*D)\vert_{Z_2}$.
This is a contradiction as $\mathcal O_P(1)\vert_{Z_i}\sim \mathcal L_i$ and $\mathcal L_1\not\equiv\mathcal L_2$.

For the second part of the statement, 
we have $\rho(P) = 2$ and so $\NE(P)$ is a cone with 2 extremal rays, one generated by the
class $f$ of a fiber of $p$.
Note that we have

\[
\OO_P(Z_2) = p^*(\mathcal{L}_1^\vee) \otimes \OO_P(1).
\]
For any curve $C \subset Z_2$, we have
\begin{equation}\label{eq:projNotProduct1}
	Z_2\cdot C = p^*(\mathcal{L}_1^\vee)\cdot C + \OO_P(1)\cdot C = \mathcal{L}_1^{\vee}\cdot p_*C + \OO_P(1)|_{Z_2} \cdot C = (\mathcal{L}_1^{\vee}\otimes \mathcal{L}_2)\cdot p_*C <0.
\end{equation}
Thus $[C]$ cannot be in the interior of $\NE(P)$ since it has negative intersection with an effective divisor. Therefore $i_*(\NE(Z_2))$ is the second extremal ray.
\end{proof}

\section{Umemura quadric fibrations}\label{sec:umemuraQuadrics}

In this section we introduce Umemura quadric fibrations and study their basic properties.

\subsection{Definition and basic properties}

Let $\mathbf{a} = (a_0,\dots,a_n)$ with $a_0,\dots,a_n \in \ZZ$.
Denote by $\mathcal{E}_\mathbf{a}$ the vector bundle $\oplus_{i=0}^n \OO_{\PP^1}(-a_i)$ over $\PP^1$.
Then the projective bundle $\PP(\mathcal{E}_\mathbf{a})$ can be described as the geometric quotient of $\left(\AA^{n+1}\setminus \{0\}\right) \times \left(\AA^2 \setminus \{0\}\right)$ by $\GG_m^2$ with the action given by:
\[
\begin{array}{ccc}
	\GG_m^2 \times \left(\AA^{n+1}\setminus \{0\}\right) \times \left(\AA^2 \setminus \{0\}\right) & \to & \left(\AA^{k+1}\setminus \{0\}\right) \times \left(\AA^2 \setminus \{0\}\right)\\
	(\lambda,\mu),(x_0,x_1,\dots,x_n;t_0,t_1) & \mapsto & (\lambda\mu^{-a_0} x_0,\lambda\mu^{-a_1} x_1, \dots ,\lambda \mu^{-a_n} x_n ; \mu t_0,\mu t_1).
\end{array}
\]

In the special case when $\mathbf{a} = (0,0,\dots,a)$ we will simply denote $\mathcal{E}_\mathbf{a}$ by $\mathcal{E}_a$.

\begin{definition}\label{def:umemuraQuadric}
	Let $n\geq 3$, $a\in \NN$ and let $g \in \kk[t_0,t_1]_{2a}$ be a homogeneous polynomial of degree $2a$. 
	We define the \emph{Umemura quadric fibration} associated to $g$ as
	\[
	\cQ_g \defeq \left\{x_1^2 - x_0x_2 + x_3^2 + \ldots +x_{n-1}^2 + g(t_0,t_1)x_n^2 = 0
	\right\}
	\subset \PP(\mathcal{E}_a).
	\]
	We will denote by $\pi\colon \cQ_g \to \PP^1$ the projection to $\PP^1$.
\end{definition}

\begin{remark}\label{rem:Ume}
	\begin{enumerate}[leftmargin=*]
		\item The choice of a non-diagonal equation is to highlight the existence of a section (see \cite{dJS}).

		\item $\cQ_g$ is rational.
		Indeed, in the open subset $\{x_2 = 1\}$ we may solve $x_1^2 - x_0x_2 + x_3^2 + \ldots + g(t_0,t_1)x_n^2=0$ for $x_0$;
		therefore the projection $(x_0:\ldots:x_n;t_0,t_1) \mapsto (x_1:\ldots:x_n;t_0,t_1)$ gives us a birational map to $\PP\left(\OO_{\PP^1}^{\oplus {n-1}} \oplus \OO_{\PP^1}(-a)\right)$, the latter being rational.
	\end{enumerate}	
\end{remark}

		 The following lemma and corollary show that Umemura quadric fibrations appear naturally as standard birational models of  quadric fibrations.

	\begin{lemma}\label{lem:genericFiber}
		Let $\KK$ be a field with $\cha(\KK) \neq 2$ and let $Q \subseteq \PP_{\KK}^n$ be a smooth quadric, $n\geq 3$.
		Assume moreover that $Q(\KK) \neq \emptyset$.
		Then up to a change of coordinates $Q$ is given by the equation
		\[
		x_1^2 - x_0x_2 + \sum_{i=3}^{k-1}x_i^2 + \sum_{i=k}^{n}\mu_i x_i^2 = 0
		\]
		and $\mu_i$ satisfying the condition that there exist no $a_{k}, \dots, a_n \in \KK$ such that $\sum_{i=k}^n a_i^2 \mu_i$ is a nonzero square in $\KK$.
	\end{lemma}
	
	\begin{proof}
		Since $Q(\KK) \neq 0$, up to a change of coordinates, we may assume that $(1:0:\dots:0) \in Q(\KK)$.
		Then $Q \cap \{x_3 = x_4 =\ldots = x_n =0\}$ is a plane quadric containing $(1:0:0)$ and, again up to change of coordinates, we may assume $Q$ is given by the equation
		\begin{equation}\label{quadrForm}\tag{$*$}
			x_1^2 - x_0x_2 + \sum_{i=3}^n l_i x_i = 0
		\end{equation}
		with $l_i \in \KK[x_0,\ldots,x_n]_1$.
		Let $M$ be the matrix of the quadratic form \eqref{quadrForm}.
		The top-left $3 \times 3$ block is given by the coefficients of $x_1^2 - x_0x_2$.
		After two changes of coordinates of the form $x_2=x_2-\lambda(x_3,\ldots,x_n)$
		and $x_1=x_1-\mu(x_3,\ldots,x_n)$, $x_0=x_0-\nu(x_3,\ldots,x_n)$ for some $\lambda,\mu,\nu\in \KK[x_0,\ldots,x_n]_1$, we may assume that there is $q\in \KK[x_3,\ldots,x_n]_2$ such that $Q$ is given by the equation $x_1^2 - x_0x_2 + q=0$.
		Thus we diagonalize $q$ and we get the desired form.
	\end{proof}
	
	\begin{proposition}
		Let $\pi\colon X \to \PP^1$ be a Mori fiber space where the generic fiber $X_{\kk(t)}$ is isomorphic to a smooth quadric hypersurface $\PP^n_{\kk(t)}$.
		Let $G = \Aut_{\kk}(X)_{\PP^1}$.
		Then $X$ is $G$-equivariantly birational to a hypersurface
		\[
		\mathcal{Q}_{\mathbf{g}} \defeq \left\{ x_1^2 - x_0x_2 + x_3^2 + \dots + x_k^2 + g_1x_{k+1}^2 + \dots + g_lx_n^2 = 0\right\} 
		\subset \PP\left(\OO_{\PP^1}^{\oplus k+1} \bigoplus_{i=1}^l \OO_{\PP^1}(-a_i)\right)
		\]
		where $a_i \in \NN$ and $g_i \in \kk[t_0,t_1]_{2a_i}$ are homogeneous polynomials of degree $2a_i$.
	\end{proposition}
	
	\begin{proof}	
		By Lemma \ref{lem:genericFiber} the generic fibre of $\pi$ is of the form
		\[
		x_1^2 - x_0x_2 + \sum_{i=3}^{k-1}x_i^2 + \sum_{i=k}^{n}\mu_i x_i^2 = 0,
		\]
		where $\mu_i = \frac{r_i}{s_i} \in \kk(t)$ for $i = k, \dots, n$.
		
		Denote by $H$ the closure of the subset $\{x_0 = 0\} \subset \pi^{-1}(U)$, where $U$ is the locus of $\PP^1$ where the $s_i$ do not vanish.
	Since 	$\pi_*\mathcal O_X(H)$ is a locally free sheaf of rank $n + 1$ (it is a torsion free sheaf over $\mathbb P^1$), it is a direct sum of line bundles $\mathcal E_{\rm{b}}=\oplus\mathcal{O}_{\PP^1}(b_i)$.
	Consider the rational map
	\[
	\xymatrix{
		X \ar[rd] \ar@{-->}[rr]^{\chi} && Y \subset \PP(\pi_*(\OO_X(H))) \ar[ld]\\
		& \PP^1,
	}
	\]
	with $Y$ the image of $X$.
	Note that, since $\rho(X/\PP^1) = 1$, $\OO_X(H)$ is $G$-invariant and so $\chi$ is $G$-equivariant.
	Over $U$ the sections $\{x_0,\ldots,x_{k-1}, s_k x_k,\ldots , s_n x_n\}$ generate
	$H^0(\pi^{-1}U,\mathcal E_{\rm{b}})$.
	Thus $\chi$ is locally given by 
	\[
	\begin{array}{ccc}
		\pi^{-1}(U) & \to & U \times \PP^n\\
		(x_0:\ldots x_n; t_0,t_1) & \mapsto & (x_0: \ldots :x_{k-1}:s_kx_k: \ldots : s_nx_n).
	\end{array}
	\]
	Taking $g_i = r_is_i \in \kk[t_0,t_1]_{2a_i}$ we recover the claimed equation and weights.
	\end{proof}

The following lemma is a computation of the Picard group and Mori cone of $\cQ_g$. We omit the proof as it follows the same lines
as \cite[Lemma 4.4.3]{BFT2}.

\begin{lemma}\label{lem:Pic}
Let $g \in \mathbb \kk[t_0, t_1]$ be a homogeneous polynomial of degree $2a$ and $\pi\colon \cQ_g \to\PP^1$ the associated Umemura quadric fibration. 
Let $F$ be a fibre of $\pi$, $H=\{x_n=0\}$, ${e}$ a curve in a fibre of $\pi$ and $\sigma$
the section of $\pi$ defined as $\sigma=\{(1:0:\ldots:0;t_0:t_1)\vert\;(t_0:t_1)\in\PP^1\}$.
Then
\begin{enumerate}
\item\label{Pic1} $\Pic(\cQ_g)=\ZZ[F]+\ZZ[H]$.
\item\label{Pic2} $\NE(\cQ_g)=\RR_+[{e}]+\RR_+[\sigma]$ and  curves with class in $\RR_+[\sigma]$ cover the divisor $H$.
\item\label{Pic3} $K_{\cQ_g} \equiv -(n-2)H + (a-2)F$.
\item\label{Pic4} The intersection numbers with the canonical divisor of $\cQ_g$ are
\[
K_{\cQ_g}\cdot{e} =n-1 \, \text{ and } \, K_{\cQ_g}\cdot\sigma =a-2.
\]
\end{enumerate}
\end{lemma}

\begin{proposition}\label{lem:Mfs}
	Let $g \in \mathbb \kk[t_0, t_1]$ be a homogeneous polynomial of degree $2a$ and $\pi\colon \cQ_g \to\PP^1$ the associated Umemura quadric fibration.
	\begin{enumerate}
		\item\label{lem:Mfs1} The singular locus of $\cQ_g$ is the discrete set
		\[
		\Sing(\cQ_g) = 
		\big\{
		(0:\ldots:0:1;t_0,t_1) \, | \, (t_0,t_1) \text{ is a multiple root of }g = 0
		\big\}.
		\]
		
		\item\label{lem:Mfs2} $\cQ_g$ has terminal singularities and is $\QQ$-factorial.
		\item\label{lem:Mfs3} $\pi\colon \cQ_g \to \PP^1$ is a Mori fiber space.
	\end{enumerate}
\end{proposition}

\begin{proof}
	Item \eqref{lem:Mfs1} follows from the Jacobian criterion. Terminality and $\QQ$-factoriality follow from \cite[Section 1.42]{Sing} and \cite[XI, Corollaire 3.14]{grothendieck2005}.

	For \eqref{lem:Mfs3}, the variety $\cQ_g$ is terminal and $\QQ$-factorial by \eqref{lem:Mfs2}.
	Moreover $\rho(\cQ_g) = 2$ by Lemma \ref{lem:Pic} and therefore $\rho(\cQ_g/\PP^1)=1$.
	Finally, let $F$ denote a general fiber of $\cQ_g \to \PP^1$.
	Then $F$ is isomorphic to a smooth quadric $Q_{n-1} \subset \PP^n$ and thus $-K_{\cQ_g}|_F = -K_F$ is ample, the equality obtained by adjunction formula.
\end{proof}

\subsection{Automorphism group}
In this subsection we will compute the automorphism group of $\cQ_g$.
We first begin by analyzing the automorphism group of the ambient space $\PP(\mathcal{E}_a)$.

\begin{lemma}\label{lem:autosBundle}
	Let $a>0$ and write $\mathcal{E}_a$ for the vector bundle $\OO_{\PP^1}^{\oplus n} \oplus \OO_{\PP^1}(-a)$. 
	Then $\Aut\big(\PP(\mathcal{E}_a)\big)_{\PP^1}$ equals
	\[
	\left\{
	\left(
	{\arraycolsep=3pt
		\begin{array}{c|c}
			\mbox{\Large $M$} 
			& 
			{\arraycolsep=1pt
				\begin{array}{c}
					0\\[-5pt]
					\vdots\\[-2pt]
					0
				\end{array}
			} 
			\\ \hline
			\resizebox{1.7cm}{!}{$\arraycolsep=1pt
				\begin{array}{ccc}
					f_0 & \dots & f_{n-1}
				\end{array}
				$}	 
			& 
			1
		\end{array}
	}
	\right)	
	\in \PGL_{n+1}\left(\kk[t_0,t_1]\right)
	\,\, \middle| \,\,
	\begin{array}{c}
		M \in \GL_n(\kk),\\[5pt]
		f_i \in \kk[t_0,t_1]_a,
	\end{array}
	\right\}
	\rtimes \GG_m/\mu_a
	,
	\]
	where $\mu_a$ denotes the group of $a$-th roots of unity, with the first factor acting on the coordinates $x_i$ and the second on the $t_i$ diagonally.
\end{lemma}
\begin{proof}
The statement follows from \cite[Proposition 2]{Gro59} and the discussion that follows that result, and some easy observations on $\Aut(\mathcal{E}_a)$.
See also \cite[Section 6.1]{BlancNotes} for a sample computation in dimension 2.
\end{proof}

\begin{proposition}\label{prop:automorphismGroups}
	Let $g \in \mathbb \kk[t_0, t_1]$ be a homogeneous polynomial of degree $2a$ and $\pi\colon \cQ_g \to\PP^1$ the associated Umemura quadric fibration.
	Then $\Autzero(\cQ_g)_{\PP^1} = \SO_n(\kk)$.

	Moreover we have
	\[
	\Autzero(\cQ_g) = 
	\left\{
	\begin{array}{ll}
		\Autzero(\cQ_g)_{\PP^1}, &\text{ if $g$ has more than $2$ roots;}\\
		\Autzero(\cQ_g)_{\PP^1} \rtimes \GG_m, &\text{ if $g$ has exactly $2$ roots;}\\
		\Autzero(\cQ_g)_{\PP^1} \times \GG_a, &\text{ if $g$ has exactly $1$ root;}\\
		\Autzero(\cQ_g)_{\PP^1} \times \PGL_2, &\text{ if $g$ is constant.}
	\end{array}
	\right.
	\]	
	In particular $H$ acts trivially on $\PP^1$ if $g$ has more than $2$ roots, and with an open orbit, otherwise.
\end{proposition}

\begin{proof}
By Lemma \ref{lem:Pic}, we have $\rho(\cQ_g/\PP^1)=1$.
	Moreover the restriction map $H^0(\PP(\mathcal E_a),\OO(1))) \to
H^0(\cQ_g,\OO(1)\vert_{\cQ_g} ))$ is surjective, and so the embedding $\cQ_g\to \PP(\mathcal E_a)$ is given by a complete linear system over $\PP^1$. It is therefore $\Autzero(\cQ_g)$-equivariant. In particular, $\Autzero(\cQ_g)_{\PP^1}$
coincides with the stabilizer of $\cQ_g$ in $\Autzero(\PP(\mathcal E_a))_{\PP^1}$.
	If $\alpha \in \Aut(\PP(\mathcal{E}_a))_{\PP^1}$,
	by Lemma \ref{lem:autosBundle} there is $M\in \GL_n(\kk)$, there are $f_i\in \kk[t_0,t_1]_a$ such that $$\alpha= \left(
	{\arraycolsep=3pt
		\begin{array}{c|c}
			\mbox{\Large $M$} 
			& 
			{\arraycolsep=1pt
				\begin{array}{c}
					0\\[-5pt]
					\vdots\\[-2pt]
					0
				\end{array}
			} 
			\\ \hline
			\resizebox{1.7cm}{!}{$\arraycolsep=1pt
				\begin{array}{ccc}
					f_0 & \dots & f_{n-1}
				\end{array}
				$}	 
			& 
			1
		\end{array}
	}
	\right)	.$$ If $\alpha$ stabilizes $\cQ_g$, we get:
	$f_i = 0$, 
	$i=0,\dots, n-1$, $M \in \Orth_n(\kk)$ 
	and the action on the $t_i$ is trivial.
	If moreover $\alpha \in \Autzero(\PP(\mathcal{E}_a))_{\PP^1}$ then $M \in \SO_n(\kk)$.
	
	For the second part, consider the short exact sequence
	\begin{equation}\label{eq:Blanchard}\tag{$\dagger$}
	0 \to \Autzero(\cQ_g)_{\PP^1} \to \Autzero(\cQ_g) \to H \to 0,
	\end{equation}
	where $H$ is the image of the homomorphism $\Autzero(\cQ_g) \to \PGL_2$ induced by Blanchard's lemma.
	Note that $\Autzero(\cQ_g)$ must permute the singular fibers, therefore $H$ is a connected group permuting the roots of $g$.
	Thus, in the four cases of the Proposition, we have that $H \leq G \leq \PGL_2$, where $G = 0, \GG_m, \GG_a$ and $\PGL_2$ respectively.
	
	If $g$ has exactly two roots then, up to change of coordinates, $g = t_0^{a_0} t_1^{a_1}$ with $0 < a_0 \leq a_1 \leq 2a$ and $a_0 + a_1 = 2a$.
	In this case we have the $\GG_m$-action
	\[
	\lambda\cdot(x_0:\ldots:x_n;t_0,t_1) \mapsto (x_0:\ldots:\lambda^{-a_0}x_n;t_0,\lambda^2t_1),
	\]
	which shows that $H = G$, and also provides a section to \eqref{eq:Blanchard}.
	
	In a similar fashion, if $g$ has $1$ root or is constant, we can write a $G$-action on $\cQ_g$, which furthermore commutes with the action of $\Autzero(\cQ_g)_{\PP^1}$, showing that the product is direct.

\end{proof}

\begin{remark}
In Proposition \ref{prop:automorphismGroups} the automorphism group only depends on the number of roots without multiplicity. Nevertheless, we do not assume that the roots of $g$ are of multiplicity one, as multiple roots naturally appear when performing the Sarkisov program, see Example \ref{ex:cancelationOfSquares}.
\end{remark}

We are now ready to compute the orbits of $\Autzero(\cQ_g)_{\PP^1}$ on $\cQ_g$.

\begin{lemma}\label{lem:orbits}
	Let $g \in \mathbb \kk[t_0, t_1]$ be a homogeneous polynomial of degree $2a$ and $\pi\colon \cQ_g \to\PP^1$ the associated Umemura quadric fibration.
	Set $H_n=\{x_n=0\}$.
	Then $\Autzero(\cQ_g)_{\PP^1}$ acts on $\cQ_g$ with the following orbits:
	\begin{enumerate}
		\item if $t \in \PP^1$ is not a root of $g$, 
		then we have the orbits $\Gamma_t \defeq \pi^{-1}(t) \cap H_n$ and its complement $\pi^{-1}(t)\setminus \Gamma_t$;
		\item if $t \in \PP^1$ is a root of $g$, then we have the orbits $p=(0:\ldots:0:1;t)$, $\Gamma_t \defeq \pi^{-1}(t) \cap H_n$ and their complement $\pi^{-1}(t)\setminus (\Gamma_t\sqcup \{p\})$.
	\end{enumerate}
\end{lemma}

\begin{proof}
	The specific description of the orbits follows from the explicit action given in Lemma \ref{lem:autosBundle}.
\end{proof}

\begin{remark}\label{rem:orbits2roots}
	When $g$ has more than $2$ roots, Lemma \ref{lem:orbits} gives a description of the orbits of $\Autzero(\cQ_g)$.
	This follows from Proposition \ref{prop:automorphismGroups}, since in that case $\Autzero(\cQ_g) = \Autzero(\cQ_g)_{\PP^1}$.
\end{remark}

\subsection{A structural result for equivariant birational maps to $\cQ_g$.}

The next two subsections consist of a collection of technical computations, that culminate in results essential to the proof of Proposition \ref{prop:termExtractionFromP}.
For the reader uninterested in the technical details, we now briefly describe the results necessary to skip ahead directly to section \ref{sec:maximality}.
In what follows $G$ denotes the group $\Autzero(\cQ_g)_{\PP^1}$.

In subsection  \ref{sec:EqRes} we produce a ``minimal'' $G$-equivariant log-resolution $X_m\to \cQ_g$ of the pair $(\cQ_g,F)$ where $F$ is a fibre of $\pi\colon \cQ_g\to\mathbb P^1$.
More specifically, in  Corollary \ref{cor:resolution}, we show that it is obtained by repeatedly blowing up the unique singular point $P$ in the fibre over $\pi(P)$.
In Proposition \ref{prop:cone} we describe the relative Mori cone $\NE(X_m/\cQ_g)$. In Corollary \ref{cor:orbitsOnLogRes} we describe the irreducible components of the preimage of $F$ in $X_m$.

In subsection \ref{sec:EqGeom} we study $G$-equivariant birational morphisms $X \to X_m$ centered over $P$.
In Lemma \ref{lem:codim2bu} we show that all orbits over $P$ have codimension at most $2$ and so, by  \cite[Proposition 2.4]{BF21}, $X \to X_m$ is a composition of smooth blow ups.
More specifically, Lemma \ref{lem:codim2bu} shows that the dual graph of the fiber over $\pi(P)$ is a chain.
We end the subsection with Proposition \ref{prop:cone bis}, by computing the relative cone of curves $\NE(X/Q_g)$, as well as some intersection numbers of its extremal rays with $K_X$.

\subsection{An equivariant resolution}\label{sec:EqRes}
In this subsection we compute an explicit $\Autzero(\cQ_g)_{\PP^1}$-equivariant resolution of singularities for $\cQ_g$.
We also we compute the orbits of the action on the resolution.

A very useful feature of the projective bundle $\PP(\mathcal{E}_a)$ is that it admits an open covering by affine spaces $U_{i,j} \defeq \{x_i = t_j = 1\} \isom \AA^{n+1}$, for $i=0,\dots,n$, $j = 0,1$.
The following lemma is a local calculation on the chart $U_{n,1}$.

\begin{lemma}\label{lem:localEquation}
	Let $\gamma(t) \in \kk[t]$ and consider the hypersurface
	\[
	U_0 \defeq \left\{ x_1^2 - x_0x_2 + x_3^2 + \dots + x_{n-1}^2 + t^k\gamma(t) = 0\right\} \subset \AA^{n+1}_{x_i,t}
	\]
	with $k \geq 0$ and $\gamma(0) \neq 0$.
	Then $U_0$ is singular at the origin if and only if $k\geq 2$.	
	
	Suppose $k \geq 2$, let $f \colon  X \to \AA^{n+1}$ be the blowup of the origin, $\tilde U_0$ the strict transform of $U_0$ and $E$ the exceptional divisor of $\tilde U_0\to U_0$.
	\begin{enumerate}
		\item\label{lem:localEquation1} There exists an open neighborhood $V \isom \AA^{n+1}$ of $X$, intersecting $E$ such that 
		\[
		U_1 \defeq \widetilde U_0 \cap V = \left\{ x_1^2 - x_0x_2 + x_3^2 + \dots + x_{n-1}^2 + t^{k-2}\gamma(t) = 0\right\}
		\]
		\item\label{lem:localEquation2} $V$ contains all the singular points of $\tilde{U}_0$ over the origin.
		\item\label{lem:localEquation3} Let $\pi\colon U_0\to \AA^1$ be the restriction of the projection $\AA^{n+1}\to\AA^1$ onto the last factor.
		Let $F$ be the fibre of $\pi$ over $0\in\AA^1$ and $\widetilde F$ its strict transform in $\widetilde U_0$. Then
		\begin{enumerate}
			\item $\widetilde F$ is an $\AA^1$-bundle over a smooth quadric in $\PP^{n-1}$; the intersection $\widetilde F\cap E$ is a section of the $\AA^1$-bundle
			\item If $k=2$ the divisor $E$ is a smooth quadric in $\PP^n$
			\item If $k\geq 3$ the divisor $E$ is a cone over  a smooth quadric in $\PP^{n-1}$. Let ${e}$ be a generator of the ruling in $E$. Then ${e}\cdot \widetilde F\vert_E=1$ and $\widetilde F\cap E$ does not contain the singular point of $E$.
		\end{enumerate}
		\item\label{lem:localEquation4} Finally, we have 
		\[
		K_{\tilde{U}_0} = f^*K_{U_0} + (n-2)E \, \text{ and } \, f^*F = \widetilde F+E.
		\]
	\end{enumerate}
\end{lemma}

\begin{proof}
	The first claim is a straightforward application of the Jacobian criterion: all partial derivatives vanish at the origin if and only if $k\geq 2$.
	
	We now proceed to the calculations on the blowup.	
	The blowup of $\AA^n$ at the origin may be described as
	\[
	\left\{
	\left((x_0,\dots,x_{n-1},t),(y_0:\ldots:y_{n-1}:s)\right) \in \AA^{n+1} \times \PP^n \, \middle| \, \rank
	\left(
	\begin{array}{ccccc}
		x_0 & x_1 & \dots & x_{n-1} & t\\
		y_0 & y_1 & \dots & y_{n-1} & s
	\end{array}
	\right)
	= 1
	\right\}.
	\]
	This is covered by the open subsets $V_i \defeq \{y_i=1\}$, $i=0,\dots,n-1$ and $V_s \defeq \{s=1\}$, all isomorphic to the affine space $\AA^{n+1}$.
	Set $V = V_s$.
	
	The strict transform $\tilde{U_0}$ is given by $\{y_1^2 - y_0y_2 + y_3^2 + \dots+ y_{n-1}^2 + t^{k-2}\gamma(t) = 0\}$ in $V$, proving (1).
	A local calculation in the open sets $V_i$, $i=0,\dots,n-1$ reveals that $\tilde{U_0}$ has no singular points over the origin there, proving (2).
	
	We now prove (3).
	The fibre $F$ is an affine quadric cone, the blow up of the vertex is a desingularisation and $\widetilde F$ is an $\AA^1$-bundle.
	Moreover $\widetilde F\cap E$ is the preimage of the vertex in $\widetilde F$ and is thus a section of the $\AA^1$-bundle. This proves (a).
	Let $\mathbb E\cong\PP^n$ be the exceptional divisor of the blow up of $\AA^{n+1}$ at the origin, with coordinates $(y_0:\ldots:y_{n-1}:s)$.
	Then equation of $E$ in $\mathbb E$	is 
	\[
	\begin{array}{cl}
		y_1^2 - y_0y_2 + y_3^2 + \dots+ y_{n-1}^2 + s^2\gamma(0)=0 & \text{ if }k=2\\
		y_1^2 - y_0y_2 + y_3^2 + \dots +y_{n-1}^2=0 & \text{ if }k\geq3.
	\end{array}
	\]
	
	Let $\mathbb F$ be the strict transform of the fibre of $\AA^{n+1}\to\AA^1$ over $0\in\AA^1$.
	Then
	\[
	{e}\cdot \widetilde F\vert_E={e}\cdot \mathbb F\vert_{\mathbb E}=1.
	\]
	The intersection $\widetilde F\cap E$ is cut out by the equation $s=0$ in $E$, proving that it does not contain the vertex $(0:\ldots:0:1)$ and concluding the proof of (b) and (c).
	
	The final claim on the pullback of the canonical divisor follows from the adjunction formula. 
	As for the pullback of $F$, we have $f^*F=\widetilde F+aE$, for some $a\geq 0$. Moreover, if ${e}$ is as in (3c), we have
	\[
	0=f^*F\cdot {e}=\widetilde F\cdot {e}+aE\cdot {e}=1-a
	\]
	and the claim follows.
\end{proof}

\begin{corollary}\label{cor:resolution}
	Let $g \in \mathbb \kk[t_0, t_1]$ be a homogeneous polynomial of degree $2a$ and $\pi\colon \cQ_g \to\PP^1$ the associated Umemura quadric fibration.
	Over every singular point $P \in \Sing(\cQ_g)$, there exists a log resolution of $(\cQ_g,F)$
	\[
	X_m \to X_{m-1} \to \ldots \to X_0 \defeq \cQ_g
	\]
	obtained by repeatedly blowing up the unique singular point over $P$ (locally described in Lemma \ref{lem:localEquation}), where $F$ is the fiber over the point $\pi(P) \in \PP^1$.
	In particular, it is $\Autzero(\cQ_g)_{\PP^1}$-equivariant.
\end{corollary}

\begin{proof}
	The existence of the resolution follows from  Lemma \ref{lem:localEquation} \eqref{lem:localEquation1} and \eqref{lem:localEquation2}.
	Indeed notice that after $i$ successive blowups $X_i \to \dots \to X_0 = \cQ_g$, $X_i$ is locally, around its unique singular point $p_i$ over $p$, given by
	\[
	\left\{ x_1^2 - x_0x_2 + x_3^2 + \dots x_{n-1}^2 + t^{k-2i}\gamma(t) = 0\right\}.
	\]
	Moreover, by Lemma \ref{lem:localEquation}\eqref{lem:localEquation3}, $p_i$ is the vertex of the quadric cone $E_i \defeq \Exc(X_i \to X_{i-1})$.
	Thus the strict transform of $E_i$ in any further blowup is smooth.
	After $m = \lceil \frac{k}{2} \rceil$ blowups $X_m$ is smooth and the preimage of $F$ is a union of smooth prime divisors meeting transversally.
	
	As for its equivariance, the action of $\Autzero(\cQ_g)_{\PP^1}$ on $\Sing(\cQ_g)$ is trivial, since the former is connected and the latter discrete.
	This proves the equivariance of the first $l = \lfloor \frac{k}{2} \rfloor$ blowups.
	If $k$ is not even, the action of $\Autzero(\cQ_g)_{\PP^1}$ on $X_l$ fixes the quadric cone $E_l$ and thus its vertex.	
	
\end{proof}

\begin{notation}\label{not:logres}
	Let $g \in \mathbb \kk[t_0, t_1]$ be a homogeneous polynomial of degree $2a$ and $\pi\colon \cQ_g \to\PP^1$ the associated $n$-dimensional Umemura quadric fibration. 
	Let $P \in \cQ_g$ be a singular point, let $F$ be the fibre over the point $\pi(P) \in \PP^1$ and let $f\colon X_m\to \cQ_g$ be the log resolution of Corollary \ref{cor:resolution}. We write $f=f_m\circ\ldots\circ f_1$ as decomposition of blowups.
	Denote by 
	\begin{itemize}
		\item $E_i$ the strict transform in $X_m$ of the exceptional divisor of $f_i$,
		\item  $\overline E_i$ the pullback of the exceptional divisor of $f_i$,
		\item ${e}_i\subseteq E_i$ for $i=1,\ldots,m-1$ the generator of the ruling of $E_i$
		\item ${e}_0\subseteq \widetilde F$ be the generator of the ruling
		\item ${e}_m$ the generator of $\NE(E_m)$ (note that $E_m$ is either isomorphic to $\PP^{n-1}$ or to $Q_{n-1}$ by Corollary \ref{cor:orbitsOnLogRes}).
	\end{itemize}
\end{notation}

\begin{proposition}\label{prop:cone}
	Notation as in \ref{not:logres}. Then
	\begin{enumerate}
		\item\label{prop:cone1} $\NE(X_m/\cQ_g)=\sum_{i=1}^m\RR_+[{e}_i]$.
		\item\label{prop:cone2} The intersections with the canonical divisor of the resolution of singularities are computed by
		\begin{equation*}
			K_{X_m} \cdot {e}_i = 
			\left\{
			\begin{array}{cl}
				-(n-2), & i = m\;\;\;\text{if m is even}\\
				-(n-1), & i = m\;\;\;\text{if m is odd}\\
				0, & i = 1,\dots,m-1\\
				-1, & i = 0.
			\end{array}
			\right.	
		\end{equation*}
	\end{enumerate}
\end{proposition}
\begin{proof}
	We prove the statement if $n\geq 5$, the proof in the case $n=4$ being similar. 
	We prove by induction on $k$ that $\NE(X_k/\cQ_g)=\sum_{i=1}^k\RR_+[{e}_i]$.
	By a slight abuse of notation we denote by ${e}_i$ the push-forward on $X_k$ of ${e}_i\subseteq X_m$.
	If $k=0$, we have $X_k=\cQ_g$ and the claim is true.
	Assume that $\NE(X_k/\cQ_g)=\sum_{i=1}^k\RR_+[{e}_i]$ and let $X_{k+1}\to X_k$ be the blow-up along the singular point.
	
	Let $C$ be an irreducible curve in $\Exc(X_{k+1}\to \cQ_g)$. Let $i$ be an integer such that $C\subseteq E_{i}$. 
	We prove by induction on $k-i+1$ that $C$ is numerically equivalent to a positive combination of ${e}_{i},\ldots,{e}_{k+1}$.
	If $i=k+1$, since $\rho(E_{k+1})=1$ the curve $C$ is numerically equivalent to a positive multiple of ${e}_{k+1}$.
	If $i<k+1$, then $E_i$ is a $\PP^1$-bundle over a smooth quadric of dimension $n-2$.
	The Mori cone of $E_i$ can be written as $\RR_+[{e}_i]+\RR_+[\gamma_i]$ where $\gamma_i$
	is contained in $E_i\cap E_{i+1}$.
	Indeed $E_i\cap E_{i+1}$ is the exceptional divisor of $f_{i+1}\vert_{E_i}$ and the curves which span it are thus extremal.
	
	Thus there are $a,b\geq 0$ such that $C\equiv a_i {e}_i+b \gamma_i$.
	The curve $\gamma_i$ is contained in $E_{i+1}$, thus, by inductive hypothesis, there are $a_{i+1},\ldots, a_{k+1}$ such that $\gamma_i\equiv a_{i+1}{e}_{i+1}\ldots+ a_{k+1}{e}_{k+1}$. Part \eqref{prop:cone1} follows.
	
	\smallskip
	
	Lemma \ref{lem:localEquation}\eqref{lem:localEquation4} implies that 
	\[
	f^*F = \tilde{F} + E_1 + \ldots + E_{m-1} + rE_m,
	\]
	where $r = 1$ if $m$ is even and $2$ otherwise.
	For $1\leq i \leq m-1$ we have
	\[
	0 = f^*F \cdot {e}_i = E_{i-1}\cdot{e}_i + E_i\cdot{e}_i + E_{i+1}\cdot{e}_i \implies E_i\cdot{e}_i = -2.
	\]
	Similarly we get $\tilde{F}\cdot{e}_0 = E_m\cdot{e}_m = -1$.
	
	Again by Lemma \ref{lem:localEquation}\eqref{lem:localEquation4} we get that
	\[
	K_{X_m} = f^*K_{\cQ_g} + \sum_{i=1}^{m-1} i(n-2)E_i + \left(s(m-1)(n-2) + t\right)E_m,
	\]
	were $(s,t) =(1, n-2)$ if $m$ is even and $(2,n-1)$ otherwise.
	Part \eqref{prop:cone2} follows by computing the intersections using the formulas above.
	
\end{proof}

\subsection{Equivariant Geometry of $\cQ_g$}\label{sec:EqGeom}
Let $g \in \mathbb \kk[t_0, t_1]$ be a homogeneous polynomial of degree $2a$ and $\pi\colon \cQ_g \to\PP^1$ the associated Umemura quadric fibration. 
Let  $P \in \cQ_g$ be a singular point, let $F$ be the fibre over the point $\pi(P) \in \PP^1$.
In Corollary \ref{cor:resolution} we provided an explicit $\Autzero(\cQ_g)_{\PP^1}$-equivariant log resolution $f\colon X_m \to \cQ_g$ of $(\cQ_g,F)$.
In this section we describe the action of $\Autzero(\cQ_g)_{\PP^1}$ on $X_m$ over $F$ as well as its action on higher models.
We also compute some intersection numbers on these higher models.

\begin{lemma}\label{lem:localEquationG}	
	Let $\gamma \in \kk[t]$ and consider the hypersurface 
	\[
	U_0 = \{x_1^2 - x_0x_2 + x_3^2 + \ldots + x_{n-1}^2 + t^k\gamma(t) =0\} \subset \AA^{n}_{x} \times \AA_t^1
	\]
	with $k\geq 2$ and $\gamma(0)\neq 0$.
	Consider the group $G\cong \SO_n(\kk)$ acting on $\AA^{n}_x$ by preserving quadratic form $x_1^2 - x_0x_2 + x_3^2 + \ldots + x_{n-1}^2$.
	Let $E_0$ be the $G$-invariant subset $\{t=0\} \cap U_0$.
	
	Let $f \colon \tilde{U_0} \to U_0$ be the blow up along the origin, $\widetilde E_0$ the strict transform of $E_0$ and $E_1$ the exceptional divisor.
	Then $f$ is $G$-equivariant and 
	\begin{enumerate}
		\item if $k=1$, $E_1 \isom \PP^{n-1}$;
		the $G$-orbits contained in $E_1$ are the smooth $(n-2)$-dimensional quadric $E_1 \cap \tilde{E_0}$ and its complement;
		\item if $k=2$, $E_1$ is an $(n-1)$-dimensional smooth quadric;
		the $G$-orbits contained in $E_1$ are the hyperplane section $E_1 \cap \tilde{E_0}$ and its complement;
		\item if $k\geq 3$, then $E_1$ is a quadric cone with vertex $P_1$;
		the $G$-orbits  contained in $E_1$ are $P_1$, the base of the cone $E_1 \cap \tilde{E_0}$ and their complement.
	\end{enumerate}	
\end{lemma}

\begin{proof}
	The action of $G$ on $E_0$ is transitive, thus $E_0$ is an orbit.
	The same is true for its strict transform $\tilde E_0$ and thus its intersection $E_1 \cap \tilde{E_0}$ with $E_1$.
	The rest is a local calculation:
	using the notation used in the first part of the proof of Lemma \ref{lem:localEquation}, the complement of $E_1 \cap \tilde{E_0}$ is $E_1 \cap V_s$;
	the description of the action of $G$ on the coordinates $y_0,\dots,y_{n-1},s$ can be deduced by its action on $x_0,\dots,x_{n-1},t$ together with the equations $x_i = ty_i$.
	
	For the second part we choose $H$ to be the subgroup of $G$ acting via
	\[
	\alpha_\lambda \cdot (x_0,x_1,x_2,x_3,\dots, x_{n-1}, t) \mapsto (\lambda x_0,x_1,\lambda^{-1}x_2,x_3,\dots, x_{n-1}, t),
	\]
	with $\lambda \in \kk^*$.
	Using the notation of Lemma \ref{lem:localEquation} we take $V = V_0 \defeq \{y_0 = 1\}$.
	There the exceptional divisor is given by $E_1 = \{h \defeq x_0 = 0\}$ and the rest follows.
\end{proof}

\begin{corollary}\label{cor:orbitsOnLogRes}
	Let $g \in \mathbb \kk[t_0, t_1]$ be a homogeneous polynomial of degree $2a$ and $\pi\colon \cQ_g \to\PP^1$ the associated Umemura quadric fibration.
	Let $P\in\cQ_g$ be the singular point of a singular fibre $F$.
	Let 
	\[
	X_m \to X_{m-1} \to \ldots \to X_0 \defeq \cQ_g
	\]
	be the log-resolution of $(\cQ_g,F)$ of Corollary \ref{cor:resolution}, with exceptional divisors $E_i \defeq \Exc(X_i \to X_{i-1})$ and $E_0 \defeq F$.	
	We have that
	\begin{enumerate}
		\item for $i< m$ the divisor $E_i$ is a $\PP^1$-bundle over a quadric $Q$ of dimension $n-2$, not isomorphic to the product $\PP^1\times Q$;
		\item for $0<i< m$ the action of $\Autzero(\cQ_g)_{\PP^1}$ on $E_i$ has exactly three orbits: the $2$ disjoint sections $E_i \cap E_{i\pm 1}$ and their complement.
		\item $E_m$ is isomorphic to $\PP^{n-1}$ if $m$ is odd and to a quadric hypersurface of dimension $n-1$ if $m$ is even;
		\item the action of $\Autzero(\cQ_g)_{\PP^1}$ on $E_m$ has exactly two orbits: $E_m \cap E_{m-1}$ and its complement.
	\end{enumerate}
\end{corollary}

\begin{proof}
	Assume that $i<m$.
	Lemma \ref{lem:localEquationG} implies that the exceptional divisor  $E_i\subseteq X_i$ of $f_i$ is a cone over a quadric $Q = Q_{n-2}$ of dimension $n-2$. Thus, its strict transforms in $X_j$ for $j>i$ are $\PP^1$-bundles over $Q$, isomorphic to $\PP\left(\OO_{Q} \oplus \OO_{Q}(-1)\right)$
	The computation of the orbits follows readily by Lemma \ref{lem:localEquationG}. 
\end{proof}

We now proceed to studying the action of $\Autzero(\cQ_g)_{\PP^1}$ on higher models of $X_m$.

\begin{lemma}\label{lem:codim2bu}
	Let $g \in \mathbb \kk[t_0, t_1]$ be a homogeneous polynomial of degree $2a$ and $\pi\colon \cQ_g \to\PP^1$ the associated $n$-dimensional Umemura quadric fibration. 
	Let $P\in\cQ_g$ be the singular point of a singular fibre $F$.
	Denote by $G = \Autzero(\cQ_g)_{\PP^1}$ and let $f\colon X_m\to \cQ_g$ be the log-resolution of $(\cQ_g,F)$ of Corollary \ref{cor:resolution}.
	
	Let 
	\[
	\xymatrix@C=.8cm{
		X_{m+\ell} \ar[r]^(.55){f_{m+\ell}} & \dots \ar[r]^(.4){f_{m+2}} & X_{m+1} \ar[r]^(.53){f_{m+1}} & X_m
	}
	\]
	be a sequence of smooth $G$-equivariant blow-ups over $P$.
	Let $Z_i$ be the centre of $f_i$, 	
	$E_i$ the strict transform of the exceptional divisor of $f_i$ for $i=1,\ldots,m+\ell$ and $E_0$ the strict transform of $F$.
	Then for every $j=1\ldots \ell$ there are integers $0 \leq h_j < k_j < m+j$ with the following properties:
	\begin{enumerate}
		\item\label{lem:codim2bu1} $Z_{m+j}=E_{h_j}\cap E_{k_j}$, it is isomorphic to a quadric of dimension $n-2$, and $G$ acts transitively on it;
		\item\label{lem:codim2bu2} for every $j > 0$ the divisor $E_{m+j}$ is the $\PP^1$-bundle $\PP\left(\OO(E_{h_j}) \oplus \OO(E_{k_j})\right)$ over $Z_{m+j}$, with $E_{h_j}\vert_{Z_{m+j}} \not\equiv E_{k_j}\vert_{Z_{m+j}} $.
		In particular $E_{m+j}$ is not a product;
		
		\item\label{lem:codim2bu3} the action of $G$ on $E_{m+j}$ has exactly three orbits: 
		two disjoint sections of the $\PP^1$-bundle corresponding to the injections 
		\[
		\OO(E_{h_j}) \hookrightarrow \OO(E_{h_j}) \oplus \OO(E_{k_j}) 
		\, \text{ and } \,
		\OO(E_{k_j}) \hookrightarrow \OO(E_{h_j}) \oplus \OO(E_{k_j})
		\]
		and their complement in $E_{j}$.	
	\end{enumerate}
\end{lemma} 
\begin{proof}
	We prove by induction on $\ell$ the following four statements:
	\begin{description}
		\item[$(1)_\ell$] on $X_{m+\ell-1}$ there exist integers $h_\ell,k_\ell$ such that $Z_{m+\ell} = E_{h_\ell}\cap E_{k_\ell}$, $Z_{m+\ell}$ is isomorphic to a quadric of dimension $n-2$, and $G$ acts transitively on it;
		
		\item[$(2)_{\ell}$] if $\ell > 1$ the divisor $E_{m+\ell}$ on $X_{m+\ell}$ is the $\PP^1$-bundle $\PP\left(\OO(E_{h_{\ell}}) \oplus \OO(E_{k_{\ell}})\right)$ over $Z_{m+\ell}$, with $E_{h}\vert_{Z_{m+\ell}} \not\equiv E_{k}\vert_{Z_{m+\ell}}$.
		In particular $E_{m+\ell}$ is not a product;
		
		\item[$(3)_{\ell}$] if $\ell > 1$  the action of $G$ on $E_{m+\ell}$ in $X_{m+\ell}$ has exactly three orbits: 
		two disjoint sections of the $\PP^1$-bundle corresponding to the injections 
		\[
		\OO(E_{h_{\ell}}) \hookrightarrow \OO(E_{h_{\ell}}) \oplus \OO(E_{k_{\ell}}) 
		\, \text{ and } \,
		\OO(E_{k_{\ell}}) \hookrightarrow \OO(E_{h_{\ell}}) \oplus \OO(E_{k_{\ell}})
		\]
		and their complement in $E_{m+\ell}$;
		
		\item[$(4)_\ell$] For every $h,k\in\{1,\ldots,m+\ell\}$ such that $E_h\cap E_k=Z\neq\emptyset$, either 
		$\OO_{Z}(E_h)$ is ample and $\OO_Z(E_k)$ is anti-ample or vice versa.
	\end{description}

	\textit{Step 1} Let $i\leq m$, and let $K_i=E_i\cap E_{i-1}$. Let $f\colon X_m\to \cQ_g$ be the log resolution of Corollary \ref{cor:resolution}. 
	We assume for simplicity that $m$ is even, the other case being analogous.
	The restriction $E_i\vert_{E_{i-1}}$ is the exceptional divisor of the blow-up $f_i$. 
	Thus   $E_i\vert_{K_i}=(E_i\vert_{E_{i-1}})\vert_{K_i}$ is anti-ample.
	Since $f^*F=\widetilde F+\sum E_j$, we have $0=f^*F\vert_{K_i}=E_i\vert_{K_i}+E_{i-1}\vert_{K_i}$.
	Therefore $E_{i-1}\vert_{K_i}$ is ample. 
	
	\medskip
	
	\noindent If $\ell=1$, the statements $(1)_1$ and $(4)_1$ follow from Corollary \ref{cor:orbitsOnLogRes} and Step 1.
	
	\medskip
	
	\noindent We assume then that $\ell>1$ and that the four statements are true for $j<\ell$.\\
	
	\noindent\textit{Step 2: proof of $(1)_\ell$ and $(2)_\ell$.} By Corollary \ref{cor:orbitsOnLogRes} and $(3)_{j}$, for every $a<\ell$
	the orbits of codimension at least 2 in $E_a\subseteq X_{m+\ell-1}$ are quadrics of the form $E_a\cap E_b$.
	This implies  $(1)_\ell$: the centre $Z_{m+\ell}$ of $f_{m+\ell}$ is of the form $E_k\cap E_h$
	for $h,k<m+\ell$.
	It also implies that the normal bundle of $Z_{m+\ell}$ in $X_{m+\ell}$ satisfies 
	\[
	N_{Z_{m+\ell}/X_{m+\ell}}=\OO(E_h)\oplus\OO(E_k)
	\]
	and that $Z_{m+\ell}$ is disjoint from all the other exceptional divisors.
	
	By Step 1 and $(4)_{\ell-1}$ without loss of generality $E_h\vert_{Z_{m+\ell}}$ is ample and 
	$E_k\vert_{Z_{m+\ell}}$ is antiample. This implies $(2)_\ell$.
	
	\noindent\textit{Step 3: proof of $(3)_\ell$.}
	The restriction $f_{m+\ell}\colon E_{m+\ell}\to Z_{m+\ell}$ gives the $\PP^1$-bundle structure.
	Let $E_h, E_k\subseteq X_{m+\ell-1}$ be such that $Z_{m+\ell}=E_k\cap E_h$.
	By abuse of notation we denote again by $E_h, E_k$ their strict transform in $X_{m+\ell}$.
	By the equivariance of all the morphisms involved, the intersections $E_h\cap E_{m+\ell}$ and $E_k\cap E_{m+\ell}$ are $G$-invariant. 
	
	We now show that there are no other closed orbits.
	Suppose by contraposition that $\widetilde Z\subseteq E_{m+\ell}$ is a closed orbit distinct from $E_h\cap E_{m+\ell}$ and $E_k\cap E_{m+\ell}$.
	By $(1)_{\ell}$ $G$ acts transitively on $Z_{m+\ell}$ and, since $f_{m+\ell}$ is $G$-equivariant, $\widetilde Z$  surjects onto $Z_{m+\ell}$
	and the restriction $f_{m+\ell}\colon \widetilde Z\to Z_{m+\ell}$ is also $G$-equivariant.
	The variety $\widetilde Z$ being an orbit, the restriction is \'etale. 
	Since $\widetilde Z, Z_{m+\ell}$ are Fano, the restriction of $f_{m+\ell}$ is an isomorphism \cite[Corollary 4.18(b)]{Debarre}.
	Thus $\widetilde Z$ is a section of the $\PP^1$-bundle. Then we get a contradiction with Lemma \ref{lem:proj not product}.
	
	\noindent\textit{Step 4: proof of $(4)_\ell$.} 
	Let $h,k\in\{1,\ldots,m+\ell\}$. If $h\neq m+\ell$ and $k\neq m+\ell$, then $f_{m+\ell}$
	is an isomorphism in a neighbourhood of $E_h\cap E_k$ and the claim follows from $(4)_{\ell-1}$.
	
	Assume now that $h=m+\ell$, set $Z=E_{m+\ell}\cap E_k$.
	In what follows we will denote by $E_i$ the strict transform of the exceptional divisor of $f_i$ in both $X_{m+\ell}$ and $X_{m+\ell-1}$.
	Notice that there is $j$ such that $Z_{m+\ell}=E_i\cap E_k$ in $X_{m+\ell-1}$.
	Set $g=f_{m+\ell}\circ\ldots\circ f_1$.
	Then there are positive integers $c_i$ such that $g^*F=\sum c_i E_i$.
	We notice that $f_{m+\ell}^*(E_i+ E_k)=E_i+ E_k+2E_{m+\ell}$ and $f_{m+\ell}^*(E_k)=E_k+E_{m+\ell}$,
	proving that $c_{m+\ell}\geq c_k+1$.
	
	By $(4)_{\ell-1}$, the restriction $E_k\vert_{Z_{m+\ell}}$ is $\pm$ample, thus
	\[
	f_{m+\ell}^*(E_k\vert_{Z_{m+\ell}})=f_{m+\ell}^*(E_k)\vert_Z=E_k\vert_Z+E_{m+\ell}\vert_Z
	\]
	is $\pm$ample.
	Moreover,
	\[
	0\sim g^*F\vert_{Z_{m+\ell}}=c_{m+\ell} E_{m+\ell}\vert_Z+  c_k E_k\vert_Z=(c_{m+\ell}-c_k) E_{m+\ell}\vert_Z+  c_k (E_k\vert_Z+E_{m+\ell}\vert_Z)
	\]
	which implies that $E_{m+\ell}\vert_Z$ is $\mp$ample and in turn that $ E_k\vert_Z$ is $\pm$ample.
	
	This concludes the proof of $(4)_\ell$.

\end{proof}

\begin{lemma}\label{lem:ampleSouth}
	
	Let $g \in \mathbb \kk[t_0, t_1]$ be a homogeneous polynomial of degree $2a$ and $\pi\colon \cQ_g \to\PP^1$ the associated Umemura quadric fibration. 
	Let $P\in\cQ_g$ be the singular point of a singular fibre $F$.
	Denote by $G = \Autzero(\cQ_g)_{\PP^1}$ and let $f\colon X_m\to \cQ_g$ be the log-resolution of $(\cQ_g,F)$ of Corollary \ref{cor:resolution}.
	
	Let 
	\[
	\xymatrix@C=.8cm{
		X_{m+\ell} \ar[r]^(.55){f_{m+\ell}} & \dots \ar[r]^(.4){f_{m+2}} & X_{m+1} \ar[r]^(.53){f_{m+1}} & X_m
	}
	\]
	be a sequence of smooth $G$-equivariant blow-ups over $P$.
	
	Let $Z\subseteq\Exc(f_1\circ\ldots\circ f_{m+\ell})$ be a codimension 2 orbit.
	Then $\Exc(f_1\circ\ldots\circ f_{m+\ell})\setminus Z$ has two connected components $\mathcal C_1,\mathcal C_2$. Assume that $E_m\cap \mathcal C_1\neq\emptyset$. Then there are two integers $N,S\in\{1,\ldots,m+\ell\}$ such that:
	\begin{itemize}
		\item $Z = E_N \cap E_S$, $E_N\cap \mathcal C_1\neq\emptyset$, $E_S\cap \mathcal C_2\neq\emptyset$;
		\item $\OO_Z(E_S)$ is ample and $\OO_Z(E_N)$ is anti-ample.
	\end{itemize}
\end{lemma}

\begin{proof}
	We prove this by induction on $\ell$.
	First suppose that $\ell = 0$.
	Then by the construction of the resolution of Corollary \ref{cor:resolution} $E_N = E_k$ and $E_S = E_{k-1}$.
	But $E_k$ is the exceptional divisor of $f_k$ and so $E_k|_Z$ is anti-ample.
	Moreover, by Lemma \ref{lem:localEquation}\eqref{lem:localEquation4}, we have $f^*F = \sum_{i=1}^{k} c_iE_i$, where $c_i = 2$, if $i = m$ and $m$ is odd, and $1$ otherwise.
	Restricting to $Z$ we get $E_{i-1}|_Z = -c_iE_i|_Z$, which is ample.
	
	We now suppose that the statement is true for all $j < \ell$, and consider the blowup $f_{m+\ell} \colon X_{m+\ell} \to X_{m + \ell -1}$ with center $Z_{m+\ell} = E_N \cap E_S$.
	Then we only need to show the statement for the two centers $Z_N \defeq E_N\cap E_{m+\ell}$ and $Z_S \defeq E_S \cap E_{m+\ell}$.
	Denote by $f_N$ the restriction $f_{m+\ell}|_{Z_N} \colon Z_N \to Z_{m+\ell}$, which is an isomorphism.
	By the inductive hypothesis we have that
	\[
	f_N^*(E_S) = \cancelto{0}{E_S|_{Z_N}} + E_{m+l}|_{Z_N}
	\]
	is ample.
	On the other hand 
	\[
	f_N^*(E_N) = E_N|_{Z_N} + E_{m+l}|_{Z_N}
	\]
	is anti-ample, which implies that $E_N|_{Z_N}$ is anti-ample.
	Restricting the class of the pullback of the fiber to $Z_N$ we get that $E_{m+\ell}|_{Z_N} = -E_N|_{Z_N}$, and is thus ample.
	The proof for $Z_S$ is analogous.
\end{proof}

\begin{remark}\label{rem:Pic1}
	In the setup of Lemma \ref{lem:codim2bu}, note that $E_m$ is isomorphic to 
	$\PP^{n-1}$ if $m$ is odd and to a quadric hypersurface of dimension $n-1$ if $m$ is even:
	indeed, for every $j$ the morphism $f_{m+j}$ is the blow-up of a smooth codimension 2 centre either contained in or disjoint from $E_m$. 
	Thus $f_{m+j}$ induces an isomorphism on $E_m$.
\end{remark}

\begin{notation}\label{not:rulings}
	Let $g \in \mathbb \kk[t_0, t_1]$ be a homogeneous polynomial of degree $2a$ and $\pi\colon \cQ_g \to\PP^1$ the associated Umemura quadric fibration. Let $P \in\cQ_g$ be the singular point of a singular fibre $F$.
	Denote by $G = \Autzero(\cQ_g)_{\PP^1}$ and let $f\colon X_m\to \cQ_g$ be the log-resolution of $(\cQ_g,F)$ of Corollary \ref{cor:resolution}.
	
	Let 
	\[
	\xymatrix@C=.8cm{
		X_{m+\ell} \ar[r]^(.55){f_{m+\ell}} & \dots \ar[r]^(.4){f_{m+2}} & X_{m+1} \ar[r]^(.53){f_{m+1}} & X_m
	}
	\]
	be a sequence of smooth $G$-equivariant blow-ups over $P$.
	Denote by 
	\begin{itemize}
		\item $E_i$ the strict transform in $X_m$ of the exceptional divisor of $f_i$,
		\item ${e}_i\subseteq E_i$ for $i\neq m$ the generator of the ruling of $E_i$
		\item ${e}_0\subseteq \widetilde F$ be the generator of the ruling
		\item ${e}_m$ the generator of $\NE(E_m)$ (recall that $E_m$ is either isomorphic to $\PP^{n-1}$ or to $Q_{n-1}$ by remark \ref{rem:Pic1}).
	\end{itemize}
\end{notation}

\begin{proposition}\label{prop:cone bis}
	Notation as in \ref{not:rulings}. Assume that $\ell>0$. 
	Assume moreover that for every $j>1$ the centre of $f_{m+j}$ lies in $E_{m+j-1}$.
	Then
	\begin{enumerate}
		\item\label{prop:cone bis1} $\NE(X_{m+\ell}/\cQ_g)=\sum_{i=1}^m\RR_+[{e}_i]$.
		\item\label{prop:cone bis2} 
		The intersections with the canonical divisor of $X_{m+\ell}$ are 
		\begin{equation*}
			K_{X_{m+\ell}} \cdot {e}_i  
			\left\{
			\begin{array}{cl}
				\geq -(n-2), & i = m\;\;\;\text{if m is even}\\
				\geq -(n-1), & i = m\;\;\;\text{if m is odd}\\
				\geq 0, & i \neq m, m+\ell\\
				= -1, & i = m+\ell.
			\end{array}
			\right.	
		\end{equation*}
	\end{enumerate}
\end{proposition}

\begin{proof}
	Assume that $n \geq 5$, the case $n = 4$ being analogous. 
	Let $C$ be an irreducible curve contained in the exceptional locus of $X_{m+\ell} \to \cQ_g$ and $k$ be such that $C \subset E_k$.
	If $k=m$ then $\rho(E_m) = 1$ and by Corollary \ref{cor:orbitsOnLogRes}, and $C \equiv a_m {e}_m$.
	Otherwise by Corollary \ref{cor:orbitsOnLogRes} and Lemma \ref{lem:codim2bu}\eqref{lem:codim2bu2} we have $\NE(E_k) = \RR_+[{e}_k] + \RR_+[\gamma_{k}]$, where $\gamma_i$ is a line in a smooth quadric of dimension $n-2$ of the form $E_i \cap E_k$.
	Thus it is enough to prove the statement for $\gamma_{k}$.
	
	Corollary \ref{cor:resolution} and Lemma \ref{lem:codim2bu} imply that each exceptional divisor with $i\neq m,1$ meets exactly two other exceptional divisors.
	Let $(F_j)_{j=1}^{m+\ell}$ be a relabeling of $(E_j)_{j=1}^{m+\ell}$ so that, for each $j$, $F_j$ meet exactly $F_{j-1}$ and $F_{j+1}$.
	We will prove by induction on $i$ that
	\[
	\gamma_{m +\ell -i} \equiv \sum_{m+\ell-i+1}^{m+\ell}a_i {e}_i,
	\]
	with $a_i \geq 0$.
	
	The base case $i= 0$ is trivial since $F_{m+{e}} = E_m$ and $\NE(F_{m+{e}}) = \RR_+[{e}_{m+\ell}]$.
	Suppose that the statement holds for all $0 \leq i \leq n$.
	By Lemma \ref{lem:ampleSouth} and Lemma \ref{lem:proj not product} $\gamma_{m +\ell -n} \subset F_{m +\ell -n} \cap F_{m +\ell -n +1}$.
	In particular  $\gamma_{m +\ell -n} \subset F_{m +\ell -n +1}$ and so there are positive numbers $\alpha,\beta$ such that
	\[
	\gamma_{m +\ell -n} \equiv \alpha{e}_{m +\ell -n +1} + \beta \gamma_{m +\ell -n+1}.
	\]
	By the inductive hypothesis $\gamma_{m +\ell -n+1}$ is a positive linear combination of the ${e}_i$, with $i > m + \ell -n +1$ and so we conclude \eqref{prop:cone bis1}.
	
	We now prove \eqref{prop:cone bis2} by induction on $\ell$.
	The base case $\ell =0$ follows from Proposition \ref{prop:cone}.
	Suppose that the statement holds for all $\ell < j$ and consider a $G$-equivariant blowup $f_{m +j} \colon X_{m+j} \to X_{m+j-1}$.
	Lemma \ref{lem:codim2bu} implies that the center $Z_{m+j}$ is of the form $E_{h_j} \cap E_{k_j}$.
	We then have $K_{X_{m+j}} = f_{m+j}^*K_{X_{m+j-1}} + E_{m+j}$, thus
	\[
	K_{X_{m+j}} \cdot {e}_i = 
	\left\{
	\begin{array}{cl}
		K_{X_{m+j-1}} \cdot {e}_i, & i\neq m+j,h_j,k_j\\
		K_{X_{m+j-1}} \cdot {e}_i + 1, & i = h_j,k_j\\
		-1, & i=m+j
	\end{array}
	\right.
	\]
	and we conclude by the inductive hypothesis.
\end{proof}

\section{Maximality of $\Autzero(\cQ_g)$}\label{sec:maximality}

In this section we study maximality of $\Autzero(\cQ_g)$ in various cases using the theory of the equivariant Sarkisov Program.

We begin with two fundamental examples.

\begin{example}\label{ex:cancelationOfSquares}
	Let $h \in \kk[t_0,t_1]$ be a homogeneous polynomial.
	The $\Autzero(\cQ_{ht_0^2})$-equivariant birational map
	\[
	\begin{array}{ccc}
		\phi\colon \cQ_{ht_0^2} & \rmap & \cQ_h\\
		(x_0:\ldots:x_n;t_0:t_1) & \mapsto & (x_0:\ldots:t_0x_n;t_0:t_1).
	\end{array}
	\]
	conjugates $\Autzero(\cQ_{ht_0^2})$ into $\Autzero(\cQ_{h})$.
	
	More specifically $\phi$ is a Sarkisov link factorizing as
	\[
	\xymatrix@R=.3cm{
		& X \ar[ld]_p \ar[rd]^q\\
		Q_{ht_0^2} \ar@{-->}[rr]^{\phi} && Q_h,
	}
	\]
	where $p$ is the blowup of the point $(0:\dots:0:1;0:1)$ and $q$ is the blowup of $\{x_n = t_0 = 0\}$.
\end{example}

\begin{example}\label{ex:linkToQuadric}
	Let $g = t_0^{a_0}t_1^{a_1}$ with $a_0+a_1$ even.
	The $\Autzero(\cQ_{g})$-equivariant morphism
	\[
	\begin{array}{ccc}
		 p\colon \cQ_{g} & \to & Q^n \subset \PP^{n+1}\\
		(x_0:\ldots:x_n;t_0:t_1) & \mapsto & (x_0:\ldots:x_{n-1}:x_nt_0^{a_0}:x_nt_1^{a_1}).
	\end{array}
	\]
	conjugates $\Autzero(\cQ_{g})$ into $\Autzero(Q^n)$
	where $Q^n$ is the smooth $n$-dimensional quadric $\{y_1^2 - y_0y_2 + y_3^2 + \dots + y_{n-1}^2 + y_ny_{n+1}=0\} \subset \PP^{n+1}$.
	
	More specifically, it is a Sarkisov link contracting $\{x_n = 0\}$ to $\Pi = \{y_n = y_{n+1} = 0\}$, and so $p$ conjugates $\Autzero(\cQ_{t_0t_1})$ into $\Autzero(Q^n;\Pi) \subsetneq \Autzero(Q^n)$.
\end{example}

\begin{proposition}\label{prop:termExtractionFromP}
Let $g \in \mathbb \kk[t_0, t_1]$ be a homogeneous polynomial of degree $2a$ and $\pi\colon \cQ_g \to\PP^1$ the associated Umemura quadric fibration.
Assume that $a\geq 2$ and let $P \in \cQ_g$ be a singular point of a fibre $F$.

Let $f\colon E \subset X \to P \in \cQ_g$ be a $G$-equivariant extremal divisorial contraction, where $G \defeq \Autzero(\cQ_g)_{\PP^1}$.
Then, after a change of coordinates, $f$ is the restriction of a standard $(1,\dots,1,b)$-blowup (see Example \ref{ex:stdWBlowup}).
\end{proposition}

\begin{proof}
Let $X_m \to Q_g$ be the log resolution of $(Q_g,F)$ of Corollary \ref{cor:resolution}.
If the valuation induced by  $E$ is not divisorial on $X_m$, then
let $W \defeq X_{m+\ell} \to X_m \to Q_g$ be the $G$-equivariant extraction of the valuation of $E$
obtained via \cite[Construction 3.1]{Kawakita}.
We thus get a birational map $\phi\colon W\dasharrow X$ that is a contraction and such that $\phi(E_{m+\ell})=E$.

\textit{Step 1.}  \emph{The map $\phi$ is a morphism.} 
The map $\phi$ restricts to a $G$-equivariant birational map $\phi\colon E_{m+\ell}\dasharrow E$.
The indeterminacy locus of the restriction of $\phi$ is both $G$-invariant and of codimension at least 2.
All the $G$-orbits have codimension at most one in $E_{m+\ell}$, thus the restriction $\phi\colon E_{m+\ell}\to E$ is a morphism.
Thus the closed orbits of $G$ in $E$ are images of the orbits of $G$ in $E_{m+\ell}$ and are either
points, quadrics of dimension $n-2$ or $\PP^1$ (the latter case occurs only if $n=4$).

Let $(p,q)\colon \widehat W\to W\times X$ be a $G$-equivariant resolution of the indeterminacies of $\phi$.
We can moreover assume that $p$ is the composition of smooth blow-ups of smooth centres.
Therefore, by Corollary \ref{cor:orbitsOnLogRes} and Lemma \ref{lem:codim2bu}, all the $p$-exceptional divisors are $\PP^1$-bundles over smooth quadrics of dimension $n-2$.
If $\phi$ is not a morphism, then there is a curve $C\subseteq \widehat W$ such that $p(C)$ is a point and $q(C)$ is a curve.
The curve $C$ is contained in the exceptional locus of $p$. Let $\widehat E$ be an irreducible component of $\Exc (p)$ such that $C\subseteq\widehat E$.
Since $p(C)$ is a point, $C$ is a fibre of the ruling defined by $p\vert_{\widehat E}$.
We set $\hat p\colon\widehat E\to Q$ the ruling defined by the restriction of $p$.
The group $G$ acts on $\widehat E$ with at least two orbits, because it preserves the intersection of $\widehat E$ with the other components of $\Exc(\widehat{W} \to \cQ_g)$.
We set $G_Q$ the kernel of the composition $G\to\Aut^{\circ}(\widehat E)\to\Aut^{\circ}(Q)$, where the last map is given by the Blanchard lemma.
Then $G_Q$ acts on the fibres of $\hat p$ with at least a fixed point, corresponding to the intersection of $\widehat E$ with the other components of $\Exc(\widehat{W} \to \cQ_g)$.

Let us consider now the restriction $q\colon \widehat E\to E$. 
Then $q$ is $G$-equivariant and 
$q(\widehat E)$ is a $G$-stable irreducible closed set in $E$.
It cannot be a point, as $C\subseteq \widehat E$ and $q(C)$ is a curve.
We assume then that $q(\widehat E)=\PP^1$. But then we must have $\widehat E=\PP^1\times Q$, contradicting Lemma \ref{lem:codim2bu}\eqref{lem:codim2bu2}. 

Assume now that $q(\widehat E)$ is a quadric $Q$ of dimension $n-2$. 
Thus the two restrictions yield an $\Autzero(Q)$-equivariant morphism $(p,q)\colon \widehat E\to Q\times Q$ which is generically finite onto its image. The image is a $G$-stable subvariety of dimension $n-1$.

Assume that $n-2\geq 3$. Then by Lemma \ref{lem:actQxQ} the group $\Autzero(Q)$ has no invariant subvariety in $Q\times Q$ of dimension $n-1$, this is a contradiction.

Assume that $n=4$. Then the image of $\widehat E$ is one of the two varieties $T_i$ from Lemma \ref{lem:actQxQ}. 
But $\Autzero(Q)$ preserves one section of $T_i\to Q$ and two sections of $\widehat E\to Q$ by Lemma \ref{lem:codim2bu} and this is a contradiction. 

\textit{Step 2.} \emph{The support of $\sum_{i=1}^{m+\ell-1}E_i$ is connected.}
Assume otherwise and let $\mathcal C_1,\mathcal C_2$ be the two connected components with $E_m\subseteq \mathcal C_1$. 
Thus, the morphism $\phi$ factors as 
\[
\xymatrix{
W \ar[r]^{\phi_2} & W_2 \ar[r]^{\phi_1} & X,
}
\]
where $\Exc(\phi_i)=\mathcal C_i$. 

Let $Z_i = \phi_2(\mathcal C_i)$ for $i = 1,2$.
Then away from $Z_1$ (resp.\ $Z_2$) $W_2$ is isomorphic to a neighborhood of $W$ (resp.\ $X$), and so $W_2$ has terminal singularities.
The relative Kodaira dimension of $W$ over $X$, and therefore of $W$ over $W_2$ is $-\infty$.
Thus $W_2$ is the output of any MMP on $W$, relative over $W_2$.
This contradicts Proposition \ref{prop:cone bis}:
indeed, by Proposition \ref{prop:cone bis}\eqref{prop:cone bis1} the first extremal contraction from $W$ is a contraction of a ray $\RR_+[{e}_i]$ with $i\neq m, m+\ell$.
By  Proposition \ref{prop:cone bis}\eqref{prop:cone bis2} those rays are not $K_W$-negative.

\textit{Step 3.} \emph{Finalizing the proof.}
By the previous step the strict transform of  $E$ in $X_{m+\ell}$ is either $E_m$, and in this case $\ell=0$,
or the unique divisor meeting the strict transform of $F$.
In the first case, we conclude as in Step 2: the morphism $\phi$ is an MMP over $X$, but the rays $\RR_+[{e}_i]$ for $i\neq m$ are not $K_W$-negative.
In the second case, the only possibility is that for every $j$ the morphism $f_{m+j}$ is the blow-up of
$E_1\cap \widetilde F$ if $j=1$ and of $E_{m+j-1}\cap \widetilde F$ if $j>0$, where $\widetilde F$ is the strict transform of $F$ in any of the $X_{m+j}$.
We are therefore in the setting of Example \ref{ex:stdWBlowup}.

\end{proof}

\begin{remark}\label{rem:termExtrFromSmoothPt}
We notice that, if $b>1$ and $P$ is a smooth point of $\cQ_g$, the extremal divisorial contraction of Proposition \ref{prop:termExtractionFromP} is an example of extremal divisorial contraction of
a divisor to a smooth point which is not a weighted blow up. Indeed, the exceptional divisor is a cone over a quadric and not a weighted projective space.
\end{remark}

\begin{corollary}\label{cor:movableTermExtr}
	Let $\cQ_g$ be an Umemura quadric fibration with and $f\colon E\subset X \to P \in \cQ_g$ be an extremal divisorial contraction, where $P$ is point of a singular fiber $F$.
	Up to a change of coordinates we may assume that $F$ is the fiber over $(0:1)$ and $g = t_0^kg'$, with $k\geq 1$ and $g'(0,1) \neq 0$.
	
	Then $ \NE(X/\PP^1)= \RR_+[e] + \RR_+[\tilde l_0]$  where $e\subseteq E$ and $\tilde l_0$ is the strict transform of a line $l_0\subseteq F$. Moreover $K_X\cdot l_0 < 0$ if and only if $k\geq 2$ and $f$ is the blowup of $\cQ_g$ along $P$.
\end{corollary}

\begin{proof}
	By Proposition \ref{prop:termExtractionFromP} we may assume that $f$ is the restriction of a standard weighted blowup of the ambient space $\PP(\mathcal{E}_a)$ with weights $(1,\dots,1,b)$.
	In that case, using the adjunction formula, we obtain
	\begin{eqnarray*}
	&-K_X = f^*(-K_{\cQ_g}) - \left(n-1 + b -\min\{k,2\}\right)E,\\
	&f^*F=\widetilde F+bE
	\end{eqnarray*}
	
	where $\widetilde F$ is the strict transform of $F$.
	Let $l_0$ be a ruling of $F$ and $\tilde{l}_0$ its strict transform in $X$.
	We first prove that $\NE(X/\PP^1)=\RR_+[e]+\RR_+[\tilde{l}_0]$, where $e\subseteq E$. 
	The ray 	$\RR_+[e]$ is extremal and $K_X$-negative. 
	By the discussion in Example \ref{ex:stdWBlowup}, the variety $\widetilde F$ is a $\PP^1$-bundle over a quadric of dimension $n-2$, and $\NE(\widetilde F)=\RR_+[e]+\RR_+[\tilde{l}_0]$.
	The intersection $\widetilde F\cdot\tilde{l}_0=-b$ can be easily computed.
	Assume that we can write $\tilde{l}_0\equiv \alpha e+\beta C$ with $C$ another curve and $\alpha,\beta\geq 0$. Then $\widetilde F\cdot C<0$ and $C\equiv \alpha' e+\beta' \tilde{l}_0$ with $\alpha',\beta'\geq 0$.
	We get $(1-\beta\beta') \tilde{l}_0\equiv(\alpha+\alpha'\beta)e$.
	Intersecting with an ample divisor, we get $1-\beta\beta'\geq 0$ and $\alpha+\alpha'\beta\geq 0$.
	Intersecting with $\widetilde F$ we get $1-\beta\beta'\leq 0$.
	We conclude that $\alpha =0$ and $C\equiv\tilde{l}_0$, proving that $\RR_+[\tilde{l}_0]$ is also extremal.
	
	Finally 
	\[
		-K_X \cdot \tilde{l}_0 = -K_{\cQ_g}\cdot l_0 - \left(n-1+b-\min\{k,2\}\right) =  \min\{k,2\}-b.
	\]
	The previous quantity being positive if and only if $k\geq 2$ and $b=1$, i.e.\ $f$ is the blowup of $\cQ_g$ along $P$.
\end{proof}

\begin{lemma}\label{lem:allLinks}
	Let $g \in \mathbb \kk[t_0, t_1]$ be a homogeneous polynomial of degree $2a$ and $\pi\colon \cQ_g \to\PP^1$ the associated Umemura quadric fibration. 
	Assume that $a \geq 2$ and $g$ has more than $2$ roots.
	Suppose that $\cQ_g \rmap Y$ is an $\Autzero(\cQ_g)$-equivariant Sarkisov link to a Mori fiber space $Y/B$.
	
	Then $Y/B = \cQ_h/\PP^1$ with 
	\[
	h = l^2g \,\text{ or }\, g = l^2h,
	\]
	for some linear polynomial $l \in \kk[t_0,t_1]_1$.
\end{lemma}

\begin{proof}
Assume that $Y/B$ is not isomorphic to $\cQ_g /\PP^1$.
Let 
\[
\xymatrix{
\cQ'\ar[d]_{\eta_1}\ar@{-->}[rr]^{\psi}&&Y'\ar[d]^{\eta_2}\\
\cQ_g\ar[d]&&Y\ar[d]\\
\PP^1\ar[dr]&&B\ar[dl]\\
&R&
}
\]
be a Sarkisov link from $\cQ_g$ to $Y$. We first prove that $\eta_1$ cannot be an isomorphism. Assume by contradiction that it is one. 
By Lemma \ref{lem:Pic} the extremal rays of $\NE(\cQ_g)$ are the extremal ray inducing $\pi$ and $\RR_+[\sigma]$ that spans a divisor and where $K_{\cQ_g}\cdot\sigma\geq 0$.
Therefore $\RR_+[\sigma]$ cannot be contracted giving rise to a divisorial contraction nor an isomorphism in codimension 1.

Thus $\eta_1\colon E\subset X \to Z\subset \cQ_g$ is an $\Autzero(\cQ_g)$-equivariant extremal divisorial contraction.
	The center $Z$ is an orbit and thus, by Lemma \ref{lem:orbits} and Remark \ref{rem:orbits2roots}, we have either
	\begin{enumerate}
		\item $Z = H_n \cap F$ for some fiber $F$ or
		\item $Z$ is the singular point $P$ of a singular fiber.
	\end{enumerate}
	
	In the first case, $Z$ has codimension $2$ in $\cQ_g$ and thus, by \cite[Proposition 2.4]{BF21}, $X \to \cQ_g$ is the blowup of $Z$.
	The resulting link is the inverse of the one in Example \ref{ex:cancelationOfSquares}, whose target is $\cQ_{h}$ with $h = gl^2$, for some linear polynomial $l \in \kk[t_0,t_1]$.

	In the latter case, by Proposition \ref{prop:termExtractionFromP}, the morphism $\eta_1$ is the restriction of a standard weighted blowup
of the ambient space $\PP(\mathcal E_a)$ with weights $(1,\ldots, 1, b)$. If $b \geq 2$ or $Z$ is not a singular point
of $\cQ_g$ Corollary \ref{cor:movableTermExtr} implies that the extremal ray $\RR_+[\tilde l_0]$ of $\NE(X/\PP^1)$, not corresponding
to $\eta_1$, is $K_X$-non-negative and span a subset of codimension 1. Therefore $\RR_+[\tilde l_0]$ cannot
be contracted giving rise to a divisorial contraction nor an isomorphism in codimension 1,
contradicting the existence of the link.
In the case that $Z$ is a singular point of $\cQ_g$ and $b = 1$ the resulting link is the one of
Example \ref{ex:cancelationOfSquares}, whose target is $\cQ_{h}$ with $g = hl^2$, for some linear polynomial $l \in \kk[t_0,t_1]$.
\end{proof}

We are now ready to prove the main result of our paper.

\begin{theorem}\label{thm:mainTheorem}
	Let $g \in \mathbb \kk[t_0, t_1]$ be a homogeneous polynomial of degree $2a$ and $\pi\colon \cQ_g \to\PP^1$ the associated Umemura quadric fibration.
	Write $g = f^2h$, where $f, h \in \kk[t_0,t_1]$ are homogeneous polynomials with $h$ being square-free. 
	Then $\Autzero(\cQ_g)$ is conjugated to a subgroup of $\Autzero(\cQ_h)$.
	
	Moreover, if $h$ and $h'$ are two square-free polynomials, we have:
	\begin{enumerate}
		\item\label{mainTh1} $\Autzero(\cQ_h)$ is a maximal connected algebraic subgroup of $\Cr_n(\kk)$  if and only if $h$ is
		constant or has at least 4 roots;
		\item\label{mainTh2}  $\Autzero(\cQ_h)$ and $\Autzero(\cQ_{h'})$ are conjugate if and only if
		$ h(t_0, t_1) = h'(\alpha(t_0, t_1))$, with $\alpha\in \PGL_2(\kk)$.
	\end{enumerate}
\end{theorem}

\begin{proof}
	The first claim follows by repeatedly applying the link in Example \ref{ex:cancelationOfSquares} to clear all square terms.
	
	Suppose now that $h$ is square free and let $G \defeq \Autzero(\cQ_h)$.
	If $h$ is constant, then $\cQ_h$ is isomorphic to the product $Q^{n-1} \times \PP^1$.
	Then $G = \Autzero(Q_{n-1}) \times \Autzero(\PP^1)$ acting factorwise, thus the action is transitive.
	This implies that there are no $G$-equivariant Sarkisov links, and so $G$ is maximal by \cite[Corollary 1.3]{Flo20}.
	
	If $h$ has exactly two roots then, up to a change of coordinates, $h = t_0t_1$.
	Example \ref{ex:linkToQuadric} shows that $G$ is conjugate to a strict subgroup of $\PSO_{n+1}(\kk)$.
	
	Finally suppose that $h$ has strictly more than $2$ roots.
	Then, by Proposition \ref{prop:automorphismGroups}, $\Autzero(\cQ_h) = \Autzero(\cQ_h)_{\PP^1} = \SO_n(\kk)$.
	Successive applications of Lemma \ref{lem:allLinks} show that if $\cQ_h$ is $G$-equivariantly birational to an Mfs $X/B$, then $X/B = \cQ_{hf^2}$.
	Since $hf^2$ has strictly more than $2$ roots too, $\Autzero(\cQ_{hf^2}) = \SO_n(\kk)$.
	Thus $G$ is maximal by \cite[Corollary 1.3]{Flo20}. This concludes \eqref{mainTh1}.
	
	Finally, let $h$ and $h'$ be two square free polynomials such that $\Autzero(\cQ_h)$ and $\Autzero(\cQ_{h'})$ are conjugate. Then there exists a birational map $\phi\colon \cQ_h\dasharrow\cQ_{h'}$. By Lemma \ref{lem:allLinks}, $\phi$ has to be an isomorphism of Mori fiber spaces, i.e. an isomorphism fitting in a diagram
	\[
	\xymatrix{
		\cQ_h\ar[r]^{\phi}\ar[d]_{\pi}&\cQ_{h'}\ar[d]^{\pi'}\\
		\PP^1\ar[r]_{\alpha}&\PP^1,
	}
	\]
	where $\alpha$ is an isomorphism.
	Since $\phi$ has to send singular fibers of $\pi$ to singular fibers of $\pi'$, and by Lemma \ref{lem:Mfs}\eqref{lem:Mfs1}, we have $h(t_0, t_1) = h'(\alpha(t_0, t_1))$. Conversely, if $h(t_0, t_1) = h'(\alpha(t_0, t_1))$ for some $\alpha\in \PGL_2(\kk)$ then we have the Mori fiber space isomorphism
	\[
	\xymatrix{
		\cQ_h\ar[r]^{\phi}\ar[d]_{\pi}&\cQ_{h'}\ar[d]^{\pi'}\\
		\PP^1\ar[r]_{\alpha}&\PP^1
	}
	\]
	that conjugates  $\Autzero(\cQ_h)$ into $\Autzero(\cQ_{h'})$.
	This concludes the proof of \eqref{mainTh2}.
\end{proof}

Finally we can deduce the following characterization of the maximality of $\Autzero(\cQ_g)$.

\begin{corollary}
	Let $g\in \mathbb \kk[t_0, t_1]$ be a homogeneous polynomial of even degree
	and $\pi\colon \cQ_g \to\PP^1$ the associated Umemura quadric fibration.
	Write $g = f^2h$ where $f, h \in \kk[t_0,t_1]$ are homogeneous polynomials with $h$ being square-free. 
	Then $\Autzero(\cQ_g)$ is a maximal connected algebraic subgroup of $\Cr_n(\kk)$  if and only if eihter $f$ and $h$ are constant or $h$ has at least 4 roots.
\end{corollary}

\begin{remark}
	When $g$ has $2$ roots, one can actually prove a more precise result: $\Autzero(\cQ_{g})$ is contained in a unique maximal connected algebraic subgroup $M$ of $\Cr_n(\kk)$; namely $M = \PSO_{n+1}(\kk)$ with the conjugation being given by Example \ref{ex:linkToQuadric}.
	
	Indeed, using the description of $\Autzero(\cQ_{g})$ of Proposition \ref{prop:automorphismGroups}, we can compute the orbits, and deduce that all equivariant Sarkisov links from $\cQ_{g}$ are of the two forms of Examples \ref{ex:cancelationOfSquares} and \ref{ex:linkToQuadric}.
	It then suffices to notice that the links of the two examples commute, i.e.\ if $g = t_0^{a_0}t_1^{a_1}$ and $g' = t_0^{b_0}t_1^{b_1}$, with $b_i = a_i \pm 2k_i$, then the diagram
	\[
	\xymatrix@R=.3cm{
	\cQ_{g} \ar[rd] \ar@{-->}[rr] && \cQ_{g'} \ar[ld]\\
	& Q^n
	}
	\]
	is commutative.
\end{remark}

\raggedright
\bibliography{bib}{}
\bibliographystyle{alpha}   

\end{document}